\documentclass[reqno]{amsart}
\usepackage{amssymb,amsmath,hyperref}
\usepackage{amsrefs}
\usepackage[foot]{amsaddr}

\newtheorem{thm}{Theorem}
\newtheorem{lem}[thm]{Lemma}
\newtheorem{cor}[thm]{Corollary}
\newtheorem{defi}[thm]{Definition}
\newtheorem{rem}[thm]{Remark}
\newtheorem{nota}[thm]{Notation}

\newtheorem{ack}[thm]{Acknowledgement}

\newtheorem*{tempo*}{Template}
\newtheorem*{rem*}{Remark}

\newtheorem{fact}{Fact}

\newcommand\be{\begin{equation}}
\newcommand\ee{\end{equation}} 

\usepackage{amsmath,amsfonts} 

\usepackage[applemac]{inputenc}

\def\bdefi{\begin{defi}\rm}
\def\edefi{\end{defi}}
\def\bnota{\begin{nota}\rm}
\def\enota{\end{nota}}
\def\FIVE{\Pi_{1}^{1}\text{-\textup{\textsf{CA}}}_{0}}

\def\ATR{\textup{\textsf{ATR}}}
\def\ZFC{\textup{\textsf{ZFC}}}
\def\IST{\textup{\textsf{IST}}}
\def\MU{\textup{\textsf{MU}}}

\def\T{\mathcal{T}}
 
\def\STP{\textup{\textsf{STP}}}

\def\FAN{\textup{\textsf{FAN}}}

\def\RWKL{\textup{\textsf{RWKL}}}
\def\H{\textup{\textsf{H}}}
\def\ef{\textup{\textsf{ef}}}
\def\ns{\textup{\textsf{ns}}}
\def\RCA{\textup{\textsf{RCA}}}
\def\({\textup{(}}
\def\){\textup{)}}
\def\WO{\textup{\textsf{WO}}}
\def\RCAo{\textup{\textsf{RCA}}_{0}^{\omega}}
\def\WKL{\textup{\textsf{WKL}}}

\def\WWKL{\textup{\textsf{WWKL}}}
\def\bye{\end{document}}
\def\P{\textup{\textsf{P}}}
\def\N{{\mathbb  N}}

\def\R{{\mathbb  R}}

\def\PC{\textup{\textsf{PC}}}

\def\m{{\mathbf  m}}

\def\FAN{\textup{\textsf{FAN}}}

\def\MUC{\textup{\textsf{MUC}}}
\def\st{\textup{st}}
\def\di{\rightarrow}
\def\asa{\leftrightarrow}
\def\ACA{\textup{\textsf{ACA}}}
\def\paai{\Pi_{1}^{0}\textup{-\textsf{TRANS}}}
\def\Paai{\Pi_{1}^{1}\textup{-\textsf{TRANS}}}
\def\QFAC{\textup{\textsf{QF-AC}}}

\def\SFF{\textup{\textsf{SFF}}}
\def\WFF{\textup{\textsf{WFF}}}
\def\UATR{\textup{\textsf{UATR}}}
\def\mSEP{\textup{\textsf{-SEP}}}

\def\con{\textup{\textsf{con}}}
\def\X{\textup{\textsf{X}}}
\def\HYP{\textup{\textsf{HYP}}}

\def\WO{\textup{\textsf{WO}}}
\def\CC{\textup{\textsf{CC}}}
\def\ADS{\textup{\textsf{ADS}}}

\def\RT{\textup{\textsf{RT}}}
\def\WT{\textup{\textsf{WT}}}
\def\EPA{\textup{\textsf{E-PA}}}
\def\EPRA{\textup{\textsf{E-PRA}}}
\def\MPC{\textup{\textsf{MPC}}}
\def\FF{\textup{\textsf{FF}}}

\def\ECF{\textup{\textsf{ECF}}}

\def\NUC{\textup{\textsf{NUC}}}
\def\LMP{\textup{\textsf{LMP}}}
\def\SCF{\textup{\textsf{SFF}}}
\def\SC{\textup{\textsf{SOC}}}
\def\SOC{\textup{\textsf{SOC}}}
\def\SOT{\textup{\textsf{SOT}}}
\def\TOT{\textup{\textsf{TOT}}}
\def\WCF{\textup{\textsf{WCF}}}
\def\HAC{\textup{\textsf{HAC}}}
\def\INT{\textup{\textsf{int}}}

\usepackage{tikz}
\usetikzlibrary{matrix, shapes.misc}

\numberwithin{equation}{section}
\numberwithin{thm}{section}

\usepackage{comment}

\begin{document}

\title[Computability theory and Nonstandard Analysis]{Computability theory, Nonstandard Analysis, and their connections}
\author{Dag Normann}
\address{Department of Mathematics, The University 
of Oslo, P.O. Box 1053, Blindern N-0316 Oslo, Norway}
\email{dnormann@math.uio.no}
\author{Sam Sanders}
\address{Department of Mathematics, TU Darmstadt, Schlossgartenstrasse 7, D-64289 Darmstadt, Germany}
\email{sasander@me.com}

\begin{abstract}
We investigate the connections between \emph{computability theory} and \emph{Nonstandard Analysis}.
In particular, we investigate the two following topics \emph{and} show that they are intimately related.  
\begin{enumerate}
\item[\textsf{(T.1)}] A basic property of \emph{Cantor space} $2^{\N}$ is \emph{Heine-Borel compactness}: for any open cover of $2^{\N}$, there is a \emph{finite} sub-cover.   
A natural question is: \emph{How hard is it to {compute} such a finite sub-cover}?  We make this precise by analysing the complexity of so-called fan functionals that given any $G:2^{\N}\di \N$, output a finite sequence $\langle f_0 , \dots, f_n\rangle $ in $2^{\N}$ such that the neighbourhoods defined from $\overline{f_i}G(f_i)$ for $i\leq n$ form a cover of $2^{\N}$.
\item[\textsf{(T.2)}]A basic property of Cantor space in \emph{Nonstandard Analysis} is Abraham Robinson's \emph{nonstandard compactness}, i.e.\ that every binary sequence is `infinitely close' to a \emph{standard} binary sequence.  We analyse the strength of this nonstandard compactness property of Cantor space, compared to the other axioms of Nonstandard Analysis and usual mathematics.
\end{enumerate}
Our study of  \textsf{(T.1)} yields exotic objects in computability theory, while \textsf{(T.2)} leads to surprising results in \emph{Reverse Mathematics}.  
We stress that \textsf{(T.1)} and \textsf{(T.2)} are highly intertwined, i.e.\ our study is `holistic' in nature in that
results in computability theory yield results in Nonstandard Analysis \emph{and vice versa}.
\end{abstract}
%

\maketitle
\thispagestyle{empty}

\section{Introduction}
We connect two seemingly unrelated fields, namely \emph{computability theory} and \emph{Nonstandard Analysis}.  
We assume basic familiarity with these fields, and the associated program \emph{Reverse Mathematics} (RM herefafter) founded by Friedman.  We refer to \cite{simpson2, simpson1} for an overview of, and \cite{stillebron} for an introduction to, RM.   
We do provide a brief introduction to Nonstandard Analysis and RM in Section \ref{base}. 
In a nutshell, we shall establish the following results.

\smallskip

Topic \textsf{(T.1)}: We study two new classes of functionals, namely the \emph{special fan functionals}, also called $\Theta$-functionals, and the (computationally weaker) \emph{weak fan functionals}, also called $\Lambda$-functionals. 
Intuitively speaking, a $\Theta$-functional computes a finite sub-cover for Cantor space from an uncountable cover, while a $\Lambda$-functional provides such a sub-cover `in the limit'.  
We show that $\Theta$ and $\Lambda$-functionals are easy to compute in Brouwer's intuitionistic\footnote{Brouwer's intuitionistic mathematics distinguishes itself from classical mathematics in that all (total) functions are continuous (see e.g.\ \cite{troeleke1}).  The functional $\exists^{2}$ is the textbook-example of a \emph{discontinuous} object (see also \cite{kohlenbach2}*{\S2}), and not available in intuititionistic mathematics.} mathematics but hard to compute in classical mathematics: the intuitionistic fan functional $\MUC$ computes $\Theta$ and $\Lambda$-functionals, but the `arithmetical comprehension' functional $\exists^{2}$ does not (and the same for \emph{any} type two functional); the classical $\exists^3$, which gives rise to \emph{full second-order arithmetic}, computes $\Theta$ and $\Lambda$-functionals.  
Thus, the latter's \emph{first-order strength} and \emph{computational hardness} diverge significantly.  
We also study the computational power of the combination of resp.\ $\Theta$ and $\Lambda$-functionals with the functional $\exists^{2}$; these combinations diverge in strength quite a lot:  
for instance, we show that the combination of any $\Theta$-functional and $\exists^{2}$ is equivalent to the higher-order version of $\ATR_{0}$, i.e.\ two relatively weak objects yield a much stronger one. 
By contrast, certain $\Lambda$-functionals do not compute `more' functions than $\exists^{2}$, i.e.\ the former are `far weaker' than $\ATR_{0}$.
 
\smallskip 
 
Topic \textsf{(T.2)}: we study the \emph{nonstandard counterparts\footnote{The principles $\STP$ and $\LMP$ are called the `nonstandard counterparts' of resp.\ $\WKL_{0}$ and $\WWKL_{0}$ in \cite{pimpson, keisler1, keisler2}. That e.g.\ $\paai$ is (or: should be) the nonstandard counterpart of arithmetical comprehension, follows from \cite{sambrouw}*{\S4.3}.  
In the latter, it is shown that the translation from \cite{brie} converts $\paai$ into arithmetical comprehension as in Feferman's mu functional.  Moreover, $\paai$ and $\Paai$ imply respectively $\ACA_{0}$ and $\FIVE$; the former also yield conservative extensions of the latter.\label{forgos}}} of the `Big Five' systems $\WKL_{0}$, $\ACA_{0}$, and $\FIVE$ of RM as follows: resp.\ the nonstandard compactness of Cantor space $\STP$ and the \emph{Transfer} axiom limited to $\Pi_{1}^{0}$-formulas $\paai$, and limited to $\Pi_{1}^{1}$-formulas $\Paai$.  While these Big Five systems are linearly ordered as $\FIVE\di \ACA_{0}\di \WKL_{0}$, we show the non-implications $\paai\not\di \STP\not\!\leftarrow \Paai$ for the respective nonstandard counterparts.  
We prove similar results for $\LMP$, the nonstandard counterpart of $\WWKL_{0}$.  By way of a surprise, we show that the combination of $\STP$ (resp. $\LMP$) with $\paai$, can (resp.\ cannot) prove $\ATR_{0}$ relative to the standard world.  
It should be noted that $\WKL_{0}$ and $\WWKL_{0}$ (and hence $\STP$ and $\LMP$) are `very close\footnote{There is no natural theorem between $\WKL_{0}$ and $\WWKL_{0}$ in the Reverse Mathematics `zoo' (\cite{damirzoo}) or the fine-grained \emph{Weihrauch degrees}, as discussed in Remark \ref{triplex}.}' in terms of logical strength. 

\smallskip

Surprising as this may seem to the uninitiated, topics \textsf{(T.1)} and \textsf{(T.2)} are intimately connected as follows: (non)computability results in \textsf{(T.1)} are obtained \emph{directly} from (non)implications in \textsf{(T.2)}, and vice versa.  In fact, $\Theta$ and $\Lambda$-functionals emerge naturally from $\STP$ and $\LMP$ when studying the computational content of Nonstandard Analysis.  Moreover, instances of the nonstandard axiom \emph{Transfer} give rise to (well-known) comprehension and choice functionals, such as the aforementioned $\exists^{2}$.  What is more, the fact that $\exists^{2}$ and any $\Theta$-functional together compute a realiser for $\ATR_{0}$ is proved (for the first time) via Nonstandard Analysis.  

\smallskip

With regard to the structure of the paper, we introduce RM and Nonstandard Analysis in Section~\ref{base}.  In Sections \ref{indie} and \ref{weak}, we introduce the special and weak fan functionals 
via specifications of their behaviour. Their basic computational properties are investigated in Sections \ref{indie}-\ref{weak}, namely that no type two functional (including the Suslin functional corresponding to $\FIVE$) can compute 
any $\Theta$ or $\Lambda$-functional in the sense\footnote{We always use Kleene's schemes S1-S9 as the meaning of `computable' in this paper, unless explicitly stated otherwise.} of Kleene's schemes S1-S9 (\cite{kleene1}; see also \cite{longmann}).  
We show that $\exists^3$, the functional corresponding to full second-order arithmetic, computes $\Theta$ 
 and $\Lambda$-functionals, while there is a $\Lambda$-functional which does not compute any $\Theta$-functional, 
even together with $\exists^{2}$.  

\smallskip

We establish in Section \ref{forgo} basic results in the RM of Nonstandard Analysis using well-known results in computability theory.  
In Section~\ref{fargp}, we establish part of the above results regarding $\STP$, $\paai$, and $\Paai$ from \textsf{(T.2)} by making heavy use of the results in Section \ref{prim}.  
For instance, negative results in \textsf{(T.1)} are used to obtain negative results in \textsf{(T.2)}.
Furthermore, we study the computational properties of $\Theta$ and $\Lambda$-functionals in detail in Section~\ref{hakkenoverdesloot}.  As we shall observe, there is a $\Lambda$-functional `closed on the hyperarithmetical', while there is no such $\Theta$-functional.  
This difference then gives rise to the following in Section \ref{weako}:  $\paai+\STP$ proves $\ATR$ relative to the standard world, while $\paai+\LMP$ does not.  
We discuss connections to Kohlenbach's generalisations of $\WKL$ in Section \ref{kokkon}.  
We summarise our results in Section \ref{kloi} and provide directions for further research.     

\smallskip

Finally, this paper connects computability theory and Nonstandard Analysis.  The first author contributed most results in the former, while the second author did so for the latter.  However, many questions were answered by translating them from one field to the other, solving them, and translating everything back, i.e.\ both authors contributed somehow to most of the paper.  As the reader will agree, our results are `holistic' in nature: results in computability theory give rise to results in Nonstandard Analysis \emph{and vice versa}.  In other words, the latter two fields turn out to be intimately connected, and this paper establishes some of these connections.  
This paper is the first of a series by the authors; the second (\cite{dagsamII}) and third (\cite{dagsamIII}) paper have also been published already.  

\section{Background: internal set theory and Reverse Mathematics}\label{base}
In this section, we introduce Nelson's axiomatic approach to Nonstandard Analysis \emph{internal set theory} (\cite{wownelly}), and it fragments based on Peano arithmetic from \cite{brie}. 
We also briefly sketch Friedman's foundational program \emph{Reverse Mathematics}.  
\subsection{Internal set theory and its fragments}\label{P}
\subsubsection{Internal set theory}\label{IIST}
In Nelson's \emph{syntactic} approach to Nonstandard Analysis (\cite{wownelly}), as opposed to Robinson's semantic one (\cite{robinson1}), a new predicate `st($x$)', read as `$x$ is standard' is added to the language of \textsf{ZFC}, the usual foundation of mathematics.  
The notations $(\forall^{\st}x)$ and $(\exists^{\st}y)$ are short for $(\forall x)(\st(x)\di \dots)$ and $(\exists y)(\st(y)\wedge \dots)$.  A formula is called \emph{internal} if it does not involve `st', and \emph{external} otherwise.   
The three external axioms \emph{Idealisation}, \emph{Standard Part}, and \emph{Transfer} govern the new predicate `st';  They are respectively defined\footnote{The superscript `fin' in \textsf{(I)} means that $x$ is finite, i.e.\ its number of elements are bounded by a natural number.} as:  
\begin{enumerate}
\item[\textsf{(I)}] $(\forall^{\st~\textup{fin}}x)(\exists y)(\forall z\in x)\varphi(z,y)\di (\exists y)(\forall^{\st}x)\varphi(x,y)$, for any internal $\varphi$.  
\item[\textsf{(S)}] $(\forall^{\st} x)(\exists^{\st}y)(\forall^{\st}z)\big((z\in x\wedge \varphi(z))\asa z\in y\big)$, for any $\varphi$.
\item[\textsf{(T)}] $(\forall^{\st}t)\big[(\forall^{\st}x)\varphi(x, t)\di (\forall x)\varphi(x, t)\big]$, where $\varphi(x,t)$ is internal, and only has free variables $t, x$.  
\end{enumerate}
The system \textsf{IST} is just \textsf{ZFC} extended with the aforementioned external axioms;  
$\IST$ is a conservative extension of \textsf{ZFC} for the internal language, as proved in \cite{wownelly}.    

\smallskip
Clearly, the extension from $\ZFC$ to $\IST$ can also be done for \emph{subsystems} of the former.   
Such extensions are studied in \cite{brie} for the classical and constructive formalisations of arithmetic, i.e.\ \emph{Peano arithmetic} and \emph{Heyting} arithmetic.  
In particular, the systems studied in \cite{brie} are \textsf{E-HA}$^{\omega}$ and $\textsf{E-PA}^{\omega}$, respectively \emph{Heyting and Peano arithmetic in all finite types and the axiom of extensionality}.       
We refer to \cite{kohlenbach3}*{\S3.3} for the exact definitions of the (mainstream in mathematical logic) systems \textsf{E-HA}$^{\omega}$ and $\textsf{E-PA}^{\omega}$.  
We introduce in Section \ref{PIPI} the system $\P$, the (conservative) extension of $\textsf{E-PA}^{\omega}$ with fragments of the external axioms of $\IST$.  

\smallskip
Finally, \textsf{E-PA}$^{\omega*}$ is the definitional extensions of \textsf{E-PA}$^{\omega}$ with types for finite sequences, as in \cite{brie}*{\S2}.  For the former system, we require some notation.  
\begin{nota}[Finite sequences]\label{skim}\rm
The systems $\textsf{E-PA}^{\omega*}$ and $\textsf{E-HA}^{\omega*}$ are definitional extensions of higher-order Peano and Heyting arithmetic with a dedicated type for `finite sequences of objects of type $\rho$', namely $\rho^{*}$.  
Since the usual coding of pairs of numbers goes through in both, there is an easy isomorphism between $0$ and $0^{*}$.  
We point out the difference between `$s^{\rho}$' and `$\langle s^{\rho}\rangle$', where the former is `the object $s$ of type $\rho$', and the latter is `the sequence of type $\rho^{*}$ with only element $s^{\rho}$'.  The empty sequence for the type $\rho^{*}$ is denoted by `$\langle \rangle_{\rho}$', usually with the typing omitted.  Furthermore, we denote by `$|s|=n$' the length of the finite sequence $s^{\rho^{*}}=\langle s_{0}^{\rho},s_{1}^{\rho},\dots,s_{n-1}^{\rho}\rangle$, where $|\langle\rangle|=0$, i.e.\ the empty sequence has length zero.
For sequences $s^{\rho^{*}}, t^{\rho^{*}}$, we denote by `$s*t$' the concatenation of $s$ and $t$, i.e.\ $(s*t)(i)=s(i)$ for $i<|s|$ and $(s*t)(j)=t(j-|s|)$ for $|s|\leq j< |s|+|t|$. For a sequence $s^{\rho^{*}}$, we define $\overline{s}N:=\langle s(0), s(1), \dots,  s(N)\rangle $ for $N^{0}<|s|$.  
For a sequence $\alpha^{0\di \rho}$, we also write $\overline{\alpha}N=\langle \alpha(0), \alpha(1),\dots, \alpha(N)\rangle$ for \emph{any} $N^{0}$.  By way of shorthand, $q^{\rho}\in Q^{\rho^{*}}$ abbreviates $(\exists i<|Q|)(Q(i)=_{\rho}q)$.  Finally, we shall use $\underline{x}, \underline{y},\underline{t}, \dots$ as short for tuples $x_{0}^{\sigma_{0}}, \dots x_{k}^{\sigma_{k}}$ of possibly different type $\sigma_{i}$.          
\end{nota}    
%
\subsubsection{The classical system $\P$}\label{PIPI}
We now introduce the system $\P$, a conservative extension of $\textsf{E-PA}^{\omega}$ with fragments of Nelson's $\IST$.  

\smallskip
To this end, we first introduce the base system $\textsf{E-PA}_{\st}^{\omega*}$.  
We use the same definition as \cite{brie}*{Def.~6.1}, where \textsf{E-PA}$^{\omega*}$ is the definitional extension of \textsf{E-PA}$^{\omega}$ with types for finite sequences as in \cite{brie}*{\S2}.  
The set $\T^{*}$ is defined as the collection of all the constants in the language of $\textsf{E-PA}^{\omega*}$.    
\bdefi\label{debs}
The system $ \textsf{E-PA}^{\omega*}_{\st} $ is defined as $ \textsf{E-PA}^{\omega{*}} + \T^{*}_{\st} + \textsf{IA}^{\st}$, where $\T^{*}_{\st}$
consists of the following axiom schemas.
\begin{enumerate}
\item The schema\footnote{The language of $\textsf{E-PA}_{\st}^{\omega*}$ contains a symbol $\st_{\sigma}$ for each finite type $\sigma$, but the subscript is essentially always omitted.  Hence $\T^{*}_{\st}$ is an \emph{axiom schema} and not an axiom.\label{omit}} $\st(x)\wedge x=y\di\st(y)$,
\item The schema providing for each closed term $t\in \T^{*}$ the axiom $\st(t)$.
\item The schema $\st(f)\wedge \st(x)\di \st(f(x))$.
\end{enumerate}
The external induction axiom \textsf{IA}$^{\st}$ states that for any (possibly external) $\Phi$:  
\be\tag{\textsf{IA}$^{\st}$}
\Phi(0)\wedge(\forall^{\st}n^{0})(\Phi(n) \di\Phi(n+1))\di(\forall^{\st}n^{0})\Phi(n).     
\ee
\edefi
Secondly, we introduce some essential fragments of $\IST$ studied in \cite{brie}.  
\bdefi[External axioms of $\P$]~
\begin{enumerate}
\item$\HAC_{\INT}$: For any internal formula $\varphi$, we have
\be\label{HACINT}
(\forall^{\st}x^{\rho})(\exists^{\st}y^{\tau})\varphi(x, y)\di \big(\exists^{\st}F^{\rho\di \tau^{*}}\big)(\forall^{\st}x^{\rho})(\exists y^{\tau}\in F(x))\varphi(x,y),
\ee
\item $\textsf{I}$: For any internal formula $\varphi$, we have
\[
(\forall^{\st} x^{\sigma^{*}})(\exists y^{\tau} )(\forall z^{\sigma}\in x)\varphi(z,y)\di (\exists y^{\tau})(\forall^{\st} x^{\sigma})\varphi(x,y), 
\]
\item The system $\P$ is $\textsf{E-PA}_{\st}^{\omega*}+\textsf{I}+\HAC_{\INT}$.
\end{enumerate}
\end{defi}
Note that \textsf{I} and $\HAC_{\INT}$ are fragments of Nelson's axioms \emph{Idealisation} and \emph{Standard part}.  
By definition, $F$ in \eqref{HACINT} only provides a \emph{finite sequence} of witnesses to $(\exists^{\st}y)$, explaining its name \emph{Herbrandized Axiom of Choice}.   

\smallskip
The system $\P$ is connected to $\textsf{E-PA}^{\omega}$ by Theorem \ref{consresultcor} below which expresses that we may obtain effective results as in \eqref{effewachten} from any theorem of Nonstandard Analysis which has the same form as in \eqref{bog}.  The scope of this theorem includes the Big Five systems of Reverse Mathematics (\cite{sambon}), the Reverse Mathematics zoo (\cite{samzooII}), and both classical and higher-order computability theory (\cite{samGH, sambon3, dagsamII}).  

\begin{thm}[Term extraction]\label{consresultcor}
If $\Delta_{\INT}$ is a collection of internal formulas and $\psi$ is internal, and
\be\label{bog}
\P + \Delta_{\INT} \vdash (\forall^{\st}\underline{x})(\exists^{\st}\underline{y})\psi(\underline{x},\underline{y}, \underline{a}), 
\ee
then one can extract from the proof a sequence of closed terms $t$ in $\mathcal{T}^{*}$ such that
\be\label{effewachten}
\textup{\textsf{E-PA}}^{\omega*} + \Delta_{\INT} \vdash (\forall \underline{x})(\exists \underline{y}\in t(\underline{x}))\psi(\underline{x},\underline{y},\underline{a}).
\ee
\end{thm}
\begin{proof}
See \cite{samGH}*{\S2} or \cite{sambon}*{\S2}.  The proof is based on the functional interpretation $S_{\st}$ from \cite{brie}.
\end{proof}
Curiously, the previous theorem is neither explicitly listed nor proved in \cite{brie}.
For the rest of this paper, the notion `normal form' shall refer to a formula as in \eqref{bog}, i.e.\ of the form $(\forall^{\st}x)(\exists^{\st}y)\varphi(x,y)$ for $\varphi$ internal.  

\smallskip
Finally, the previous theorems do not really depend on the presence of full Peano arithmetic.  
We shall study the following weaker systems.   
\bdefi[Higher-order Reverse Mathematics]\label{rebs}~
\begin{enumerate}
\item Let \textsf{E-PRA}$^{\omega}$ be the system defined in \cite{kohlenbach2}*{\S2} and let \textsf{E-PRA}$^{\omega*}$ 
be its definitional extension with types for finite sequences as in \cite{brie}*{\S2}. 
\item $(\QFAC^{\rho, \tau})$ For every quantifier-free internal formula $\varphi(x,y)$, we have
\be\label{keuze}
(\forall x^{\rho})(\exists y^{\tau})\varphi(x,y) \di (\exists F^{\rho\di \tau})(\forall x^{\rho})\varphi(x,F(x))
\ee
\item The system $\RCAo$ is $\textsf{E-PRA}^{\omega}+\QFAC^{1,0}$.  
\end{enumerate}
\edefi
The system $\RCAo$ is Kohlenbach's `base theory of higher-order Reverse Mathematics' as introduced in \cite{kohlenbach2}*{\S2}.  
We permit ourselves a slight abuse of notation by also referring to the system $\textsf{E-PRA}^{\omega*}+\QFAC^{1,0}$ as $\RCAo$.
\begin{cor}\label{consresultcor2}
The previous theorem and corollary go through for $\P$ and $\textsf{\textup{E-PA}}^{\omega*}$ replaced by $\P_{0}\equiv \textsf{\textup{E-PRA}}^{\omega*}+\T_{\st}^{*} +\HAC_{\INT} +\textsf{\textup{I}}+\QFAC^{1,0}$ and $\RCAo$.  
\end{cor}
\begin{proof}
The proof of \cite{brie}*{Theorem 7.7} goes through for any fragment of \textsf{E-PA}$^{\omega{*}}$ which includes \textsf{EFA}, sometimes also called $\textsf{I}\Delta_{0}+\textsf{EXP}$.  
In particular, the exponential function is (all what is) required to `easily' manipulate finite sequences.    
\end{proof}
We note that Ferreira and Gaspar present a system similar to $\P$ in \cite{fega}, which however is less suitable for our purposes.    

\smallskip

Finally, we discuss the exact connection between our systems of Nonstandard Analysis and computability theory provided by Theorem \ref{consresultcor}.  
The crucial point here is that in the syntactic theory of Nonstandard Analysis, the usual quantifiers $\exists$ and $\forall$ play the role of `uniform quantifiers' (see \cite{uhberger}) which are \emph{ignored} by the functional interpretation $S_{\st}$ from \cite{brie} used in the proof of Theorem \ref{consresultcor}, while the standard quantifiers $\exists^{\st}$ and $\forall^{\st}$ are given computational meaning.  
Indeed, the functional interpretation $S_{\st}$ applied to the proof of \eqref{bog} yields a term $t(\underline{x})$ in which the $(\forall^{\st}\underline{x})$ quantifier in \eqref{bog} describes the input variables, while the $(\exists^{\st}\underline{y})$ quantifier describes the output variables. 
This gives each of the nonstandard axioms a clear computational meaning entirely independent of Nonstandard Analysis per se, which may be of comfort to some who find Nonstandard Analysis alien. 
Those interested in this kind of development should consult \cite{sambrouw}.
The following remark provides more intuition regarding applications of Theorem \ref{consresultcor}.
\begin{rem*}[Using term extraction]\label{doeisnormaal}\rm
First of all, term extraction as in Theorem~\ref{consresultcor} is restricted to normal forms.  We now show that normal forms are `closed under implication', as follows. 
Let $\varphi, \psi$ be internal and consider the following implication between normal forms:
\be\label{nora}\tag{N1}
(\forall^{\st}x)(\exists^{\st}y)\varphi(x, y)\di (\forall^{\st}z)(\exists^{\st}w)\psi(z, w).  
\ee
Since standard functionals have standard output for standard input, \eqref{nora} implies
\be\label{nora2}\tag{N2}
(\forall^{\st}\zeta)\big[(\forall^{\st}x)\varphi(x, \zeta(x))\di (\forall^{\st}z)(\exists^{\st}w)\psi(z, w)\big].  
\ee
Bringing all standard quantifiers outside, we obtain the following normal form:
\be\label{nora3}\tag{N3}
(\forall^{\st}\zeta, z)(\exists^{\st} w, x)\big[\varphi(x, \zeta(x))\di \psi(z, w)\big],
\ee
as the formula in square brackets is internal.  Now, \eqref{nora3} is equivalent to \eqref{nora2}, but one usually weakens the latter as follows:  
\be\label{nora4}\tag{N4}
(\forall^{\st}\zeta, z)(\exists^{\st} w)\big[(\forall x)\varphi(x, \zeta(x))\di \psi(z, w)\big],
\ee
as \eqref{nora4} is closer to the usual mathematical definitions.  

\medskip
\noindent
Secondly, assuming \eqref{nora} is provable in $\P$, so is \eqref{nora4} and we obtain a term $t$ with
\be\label{nora5}\tag{N5}
(\forall \zeta, z)(\exists w\in t(\zeta, z))\big[(\forall x)\varphi(x, \zeta(x))\di \psi(z, w)\big]
\ee
being provable in $\EPA^{\omega*}$.  We now omit the term $t$ and bring all quantifiers inside again, yielding that $\EPA^{\omega*}$ proves:
\be\label{nora6}\tag{N6}
(\exists \zeta)(\forall x)\varphi(x, \zeta(x))\di  (\forall z)(\exists w)\psi(z, w).
\ee
Finally, we shall often shorten the below proofs by just providing normal forms and jumping straight from \eqref{nora} to \eqref{nora5} or \eqref{nora6} whenever possible.  
 \end{rem*}  

\subsubsection{Notations and conventions}
We introduce notations and conventions for $\P$.  

\smallskip
First of all, we mostly use the same notations as in \cite{brie}.  
\begin{rem}[Notations]\label{notawin}\rm
We write $(\forall^{\st}x^{\tau})\Phi(x^{\tau})$ and $(\exists^{\st}x^{\sigma})\Psi(x^{\sigma})$ as short for 
$(\forall x^{\tau})\big[\st(x^{\tau})\di \Phi(x^{\tau})\big]$ and $(\exists x^{\sigma})\big[\st(x^{\sigma})\wedge \Psi(x^{\sigma})\big]$.     
A formula $A$ is `internal' if it does not involve $\st$.  The formula $A^{\st}$ is defined from $A$ by appending `st' to all quantifiers (except bounded number quantifiers).    
\end{rem}
\smallskip
Secondly, we use the usual extensional notion of equality. 
\begin{rem}[Equality]\label{equ}\rm
The system $\textsf{E-PA}^{\omega*}$ includes equality between natural numbers `$=_{0}$' as a primitive.  Equality `$=_{\tau}$' and inequality $\leq_{\tau}$ for $x^{\tau},y^{\tau}$ is:
\be\label{aparth}
[x=_{\tau}y] \equiv (\forall z_{1}^{\tau_{1}}\dots z_{k}^{\tau_{k}})[xz_{1}\dots z_{k}=_{0}yz_{1}\dots z_{k}],
\ee
\be\label{aparth1}
[x\leq_{\tau}y] \equiv (\forall z_{1}^{\tau_{1}}\dots z_{k}^{\tau_{k}})[xz_{1}\dots z_{k}\leq_{0}yz_{1}\dots z_{k}],
\ee
if the type $\tau$ is composed as $\tau\equiv(\tau_{1}\di \dots\di \tau_{k}\di 0)$.
In the spirit of Nonstandard Analysis, we define `approximate equality $\approx_{\tau}$' as follows (with the type $\tau$ as above):
\be\label{aparth2}
[x\approx_{\tau}y] \equiv (\forall^{\st} z_{1}^{\tau_{1}}\dots z_{k}^{\tau_{k}})[xz_{1}\dots z_{k}=_{0}yz_{1}\dots z_{k}]
\ee
All the above systems include the \emph{axiom of extensionality} for all $\varphi^{\rho\di \tau}$ as follows:
\be\label{EXT}\tag{\textsf{E}}  
(\forall  x^{\rho},y^{\rho}) \big[x=_{\rho} y \di \varphi(x)=_{\tau}\varphi(y)   \big].
\ee
However, as noted in \cite{brie}*{p.\ 1973}, the so-called axiom of \emph{standard} extensionality \eqref{EXT}$^{\st}$ is problematic and cannot be included in $\P$ or $\P_{0}$.  
\end{rem}
Thirdly, $\P$ and $\P_{0}$ prove the \emph{overspill principle}, which expresses that no internal formula captures the standardness predicate exactly.  
\begin{thm}\label{doppi}
The systems $\P$ and $\P_{0}$ prove \emph{overspill}, i.e.\ for any internal $\varphi$:
\be\tag{\textsf{OS}}
(\forall^{\st}x^{\rho})\varphi(x)\di (\exists y^{\rho})\big[\neg\st(y)\wedge \varphi(y)  \big],
\ee
\end{thm}
\begin{proof}
See \cite{brie}*{Prop.\ 3.3}.  
\end{proof}
\begin{rem}[Using $\HAC_{\INT}$ and $\textsf{I}$]\label{simply}\rm
By definition, $\HAC_{\INT}$ produces a functional $F^{\sigma\di \tau^{*}}$ which outputs a \emph{finite sequence} of witnesses.  
However, $\HAC_{\INT}$ provides an actual \emph{witnessing functional} assuming (i) $\tau=0$ in $\HAC_{\INT}$ and (ii) the formula $\varphi$ from $\HAC_{\INT}$ is `sufficiently monotone' as in: 
$(\forall^{\st} x^{\sigma},n^{0},m^{0})\big([n\leq_{0}m \wedge\varphi(x,n)] \di \varphi(x,m)\big)$.    
Indeed, in this case one simply defines $G^{\sigma+1}$ by $G(x^{\sigma}):=\max_{i<|F(x)|}F(x)(i)$ which satisfies $(\forall^{\st}x^{\sigma})\varphi(x, G(x))$.  
To save space in proofs, we will sometimes skip the (obvious) step involving the maximum of finite sequences, when applying $\HAC_{\INT}$.  
We assume the same convention for terms obtained from Theorem \ref{consresultcor}, and applications of the contraposition of idealisation \textsf{I}.  
\end{rem}

\subsection{Introducing Reverse Mathematics}\label{RM}
Reverse Mathematics (RM) is a program in the foundations of mathematics initiated around 1975 by Friedman (\cites{fried,fried2}) and developed extensively by Simpson (\cite{simpson2}) and others.  We refer to \cite{simpson2, simpson1} for an overview of RM, and to \cite{stillebron} for a layman introduction; we do sketch some aspects of RM essential to this paper.  

\smallskip  
The aim of RM is to find the axioms necessary to prove a statement of \emph{ordinary} mathematics, i.e.\ dealing with countable or separable objects.   
The classical\footnote{In \emph{Constructive Reverse Mathematics} (\cite{ishi1}), the base theory is based on intuitionistic logic.} base theory $\RCA_{0}$ of `computable\footnote{The system $\RCA_{0}$ consists of induction $I\Sigma_{1}$, and the {\bf r}ecursive {\bf c}omprehension {\bf a}xiom $\Delta_{1}^{0}$-CA.} mathematics' is always assumed.  
Thus:
\begin{quote}
\emph{The aim of \emph{RM} is to find the minimal axioms $A$ such that $\RCA_{0}$ proves $ [A\di T]$ for statements $T$ of ordinary mathematics.}
\end{quote}
Surprisingly, once the minimal axioms $A$ have been found, we almost always also have $\RCA_{0}\vdash [A\asa T]$, i.e.\ not only can we derive the theorem $T$ from the axioms $A$ (the `usual' way of doing mathematics), we can also derive the axiom $A$ from the theorem $T$ (the `reverse' way of doing mathematics).  In light of these `reversals', the field was baptised `Reverse Mathematics'.  \smallskip
Perhaps even more surprisingly, in the majority\footnote{Exceptions are classified in the so-called Reverse Mathematics zoo (\cite{damirzoo}).
} 
of cases, for a statement $T$ of ordinary mathematics, either $T$ is provable in $\RCA_{0}$, or the latter proves $T\asa A_{i}$, where $A_{i}$ is one of the logical systems $\WKL_{0}, \ACA_{0},$ $ \ATR_{0}$ or $\FIVE$.  The latter together with $\RCA_{0}$ form the `Big Five' and the aforementioned observation that most mathematical theorems fall into one of the Big Five categories, is called the \emph{Big Five phenomenon} (\cite{montahue}*{p.\ 432}).  
Furthermore, each of the Big Five has a natural formulation in terms of (Turing) computability (see e.g.\ \cite{simpson2}*{I.3.4, I.5.4, I.7.5}).
As noted by Simpson in \cite{simpson2}*{I.12}, each of the Big Five also corresponds (sometimes loosely) to a foundational program in mathematics.  

\smallskip

Now, the logical framework for RM is \emph{second-order arithmetic}, i.e.\ only natural numbers and sets thereof are available.  For this reason higher-order objects such as $\R\di \R$-functions and topologies are not available directly.  For instance, continuous functions are represented in RM by so-called \emph{codes} (see e.g.\ \cite{simpson2}*{II.6.1} and \cite{mummy}), while \emph{discontinuous} functions are represented by sequences of such codes (\cite{simpson2}*{X.1}).  
Kohlenbach shows in \cite{kohlenbach4}*{\S4} that the use of codes to represent continuous functions 
does not affect the RM of $\WKL_{0}$.   
He has also introduced \emph{higher-order} RM in which discontinuous functions \emph{are} present (see \cite{kohlenbach2}*{\S2} and Definition \ref{rebs}).    
The authors show in \cite{dagsamVI} that the use of codes in measure theory does have a major impact on the logical strength of basic convergence theorems, and hence RM.  

\smallskip

Finally, we consider an interesting observation regarding the Big Five systems of Reverse Mathematics, namely that these five systems satisfy the strict implications:
\be\label{linord}
\FIVE\di \ATR_{0}\di \ACA_{0}\di\WKL_{0}\di \RCA_{0}.
\ee
By contrast, there are many incomparable logical statements in second-order arithmetic.  For instance, a regular plethora of such statements may be found in the \emph{Reverse Mathematics zoo} in \cite{damirzoo}.  The latter is intended as a collection of theorems which fall outside of the Big Five classification of RM.  
As detailed in Section~\ref{weako}, special fan functionals do not fit into the usual hierarchy of (higher-order) RM.

\section{Special fan functionals and their computational properties}\label{prim}
In this section, we study the relationship between the new \emph{special} and \emph{weak} fan functionals and existing functionals like $\exists^{2}$.
As a main result, we show that the latter (and in fact any type two functional) cannot compute any special or weak fan functional, in the sense of Kleene's S1-S9 (see \cites{kleene1, longmann} for the latter).    

\smallskip

As to their provenance, special fan functionals were first introduced in \cite{samGH}*{\S3} in the study\footnote{In a nutshell, the Gandy-Hyland functional $\Gamma$ is defined `in terms of itself' as follows: $\Gamma(Y^{2},s^{0^{*}} ):= Y(s *0*\lambda n.\Gamma(Y, s*n))$.  It is shown in \cite{samGH} that one can replace this definition by a primitive recursive one \emph{involving nonstandard numbers}.  In particular, one `only' needs to apply the definition of $\Gamma$ for $N$-many times, for nonstandard $N$, to obtain a primitive recursive functional.  One uses $\STP$ defined below to obtain this result, and applying term extraction as in Theorem \ref{consresultcor} then yields that $\Gamma$ is computable in terms of any $\Theta$-functional and other functionals.\label{hitherto}} of the Gandy-Hyland functional (\cite{gandymahat}).  
Special fan functionals are part of classical mathematics in that they can be defined in a (relatively strong) fragment of set theory (corresponding to full second-order arithmetic) by Theorem~\ref{import3} in Section~\ref{tokie}.  
Furthermore, special fan functionals may be obtained from the \emph{intuitionistic} fan functional, as shown in Section \ref{indie}.  
This result shows that the existence of a special fan functional has quite \emph{weak first-order strength} in contrast to its aforementioned considerable \emph{computational hardness} in classical mathematics.       

\smallskip

Finally, to show that special fan functionals are not an `isolated accident', we introduce the (strictly) weaker class of \emph{weak fan functionals} in Section~\ref{weak}.  
Intuitively speaking, special fan functionals are based on (a nonstandard version of) $\WKL_{0}$ from Section~\ref{RM} while weak fan functionals are based on (a nonstandard version of) the weaker $\WWKL_{0}$, also introduced in Section~\ref{weak}.  
It should be noted that our below definition of the fan functionals is \emph{different} from the (original) definition used in e.g.\ \cite{samGH}.   
That these definitions are equivalent is shown in \cite{dagsamII}*{\S2.6}.

\subsection{The special and intuitionistic fan functionals}\label{indie}
In this section, we define the class of \emph{special fan functionals}, also called $\Theta$-functionals, and show that the intuitionistic fan functional can compute special fan functionals. 
In particular, the name `special fan functional' derives from this relative computability result. 

\smallskip

Intuitively, any $\Theta$-functional outputs a finite sub-cover on input an uncountable cover of $2^{\N}$.  
We usually simplify the type of these fan functionals to `$3$'.  We reserve the symbol $\Theta$ to denote special fan functionals.  
It goes without saying that $\Theta$-functionals are not unique: just add extra binary sequences to the finite sub-cover.  

\medskip

We write `$f\in [\sigma]$' for $\overline{f}|\sigma|=_{0^{*}}\sigma$, where $\tau^{*}$ is the type of finite sequences of type $\tau$ objects.  For $w^{\tau^{*}}=\langle t_{0}, \dots, t_{k}\rangle$, we write $|w|=k+1$ and $w(i)=t_{i}$ for $i<|w|$. 
These `finite sequence' notations are discussed in detail in Notation \ref{skim}.
\bdefi[Special fan functionals]\label{dodier}
$\SFF(\Theta)$ is as follows for $\Theta^{2\di 1^{*}}$:
\be\label{kijkma}
(\forall G^{2})(\forall f^{1}\leq1)(\exists g\in \Theta(G))(f\in [\overline{g}G(g)]).
\ee
Any functional $\Theta$ satisfying $\SFF(\Theta)$ is referred to as a \emph{special fan functional}.
\edefi
Following \eqref{kijkma}, any functional $G^{2}$ gives rise to a `canonical cover' $\cup_{f\in 2^{\N}}[\overline{f}G(f)]$ of Cantor space, and $\Theta(G)$ is a finite sub-cover thereof, i.e.\ $\cup_{i\leq k}[\overline{f_{i}}G(f_{i})]$ also covers $2^{\N}$ in case $\Theta(G)=\langle f_{0}, \dots, f_{k}\rangle$.
Note that Cousin (\cite{cousin1}) and Lindel\"of (\cite{blindeloef}) make use of such canonical covers (for $\R^{n}$) rather than the modern/general notion of cover.
In light of \eqref{kijkma}, special fan functionals may be called `realisers for the Heine-Borel theorem or Cousin lemma for $C$'.

\smallskip

We stress that $G^{2}$ in $\SCF(\Theta)$ may be \emph{discontinuous} and that Kohlenbach has argued for the study of discontinuous functionals in higher-order RM (see Section~\ref{RM}).  As it turns out, $\Theta$-functionals are intimately connected to Tao's notion of \emph{metastability}, as explored in \cite{samFLO}.  

\smallskip

Secondly, we define the \emph{intuitionistic fan functional} $\Omega^{3}$ (see \cite{kohlenbach2}*{\S3} and \cite{troelstra1}*{2.6.6}).  Note that combining the latter with a discontinuous functional like $\exists^{2}$ leads to a contradiction.   
\be\tag{$\MUC(\Omega)$}
(\forall Y^{2}) (\forall f^{1}, g^{1}\leq_{1}1)(\overline{f}\Omega(Y)=\overline{g}\Omega(Y)\notag \di Y(f)=Y(g)).   
\ee
There are a number of equivalent formulations of the intuitionistic fan functional (e.g.\ outputting a supremum for every $Y^{2}$ on Cantor space rather than a modulus of uniform continuity), corresponding to the RM-equivalences from \cite{simpson2}*{IV.2.3}.    

\smallskip

As to the logical strength of $(\exists \Omega^{3})\MUC(\Omega)$, the latter yields a conservative extension of $\WKL_{0}$ by the following theorem, where `$\RCA_{0}^{2}$' is just the base theory $\RCA_{0}$ formulated with function variables (see \cite{kohlenbach2}*{\S2} for details and definitions).  
\begin{thm}\label{proto}
$\RCA_{0}^{\omega}+(\exists \Omega^{3})\MUC(\Omega)$ is a conservative extension of $\RCA_{0}^{2}+\WKL$.  
\end{thm}
\begin{proof}
As suggested in \cite{kohlenbach2}*{\S3}, one can modify the proofs in \cite{troelstra1}*{\S2.6} to establish the conservation result in the theorem, but it seems worthwhile to discuss some details.    
Indeed, in the latter reference, the so-called $\ECF$-interpretation is defined which, intuitively speaking, replaces all higher-order functionals (of type two or higher) by type one \emph{codes} (in the sense of Reverse Mathematics).      

\smallskip

Now, the $\ECF$-interpretation of $(\exists \Omega^{3})\MUC(\Omega)$ expresses that there is a code $\alpha^{1}$ which yields a modulus of uniform continuity on Cantor space on input a code $\beta^{1}$ representing an (automatically continuous) type two functional.  
As follows from the discussion in \cite{longmann}*{p.\ 459}, the $\ECF$-interpretation of $(\exists \Omega^{3})\MUC(\Omega)$ is equivalent to weak K\"onig's lemma.  Alternatively, one can explicitly define the aforementioned code $\alpha^{1}$ and show that it has the required properties using the contraposition of $\WKL$, as done in \cite{troelstra1}*{2.6.6} and \cite{noortje}*{p.\ 101}.  
\end{proof}
We note that the $\ECF$-interpretation is related to function realizability as in Kleene-Vesley (\cite{KV}).
Next, recall that the fan theorem $\FAN$  is the classical contraposition of $\WKL$, as follows:
\be\tag{$\FAN$}
(\forall T \leq_{1}1)\big[ (\forall \beta\leq_{1}1)(\exists m)(\overline{\beta}m\not \in T)\di (\exists k^{0})(\forall \beta\leq_{1}1)(\exists i\leq k)(\overline{\beta}i\not\in T) \big]. 
\ee
We also introduce the `effective version' of the fan theorem as follows.
\bdefi[Effective fan theorem]
\be\tag{$\FAN_{\ef}(h)$}
(\forall T^{1}\leq_{1}1, g^{2})\big[ (\forall \alpha\leq_{1}1)(\overline{\alpha}g(\alpha)\not\in T)\di (\forall \beta\leq_{1}1)(\overline{\beta}h(g, T)\not\in T)   \big].
\ee
\edefi
Clearly, the existence of $h$ as in the effective fan theorem implies $\FAN$ in $\RCAo$.  
Furthermore, with a further minimum of the axiom of choice $\QFAC^{2,1}$, the latter also follows from the former.  
We have the following theorem. 
\begin{thm}\label{kinkel}
There are terms $s^{3\di 3},t^{3\di 3}$ such that $\textsf{\textup{E-PA}}^{\omega*}$ proves:
\be\label{ikkeltje}
(\forall \Omega^{3})(\MUC(\Omega)\di \SCF(t(\Omega))) \wedge (\forall \Theta^{3})(\SCF(\Theta)\di \FAN_{\ef}(s(\Theta))).
\ee
\end{thm}
\begin{proof}
The second part is immediate.  
For the first part, let $\Omega$ be as in $\MUC(\Omega)$ and define $\Theta(g)$ to consist of the finite sequence of binary sequences $\tau*00\dots$ where $|\tau|=k_{0}\wedge \tau\leq_{0^{*}}1$ for $k_{0}:=\max_{|\sigma|=\Omega(g)\wedge \sigma\leq_{0^{*}}1}g(\sigma*00\dots)$.
Since $g$ in uniformly continuous on $2^{\N}$ by $\MUC(\Omega)$, we clearly have $\SFF(\Theta)$.
\end{proof}
The previous proof seems to go through in constructive mathematics. 
\begin{cor}\label{corhoeker}
$\RCA^{\omega}_{0}+(\exists \Theta^{3})\SCF(\Theta)$ is a conservative extension of $\RCA_{0}^{2}+\WKL$.   
\end{cor}
\begin{proof}
Immediate by combining the theorem and Theorem \ref{proto}.  Alternatively, one readily verifies that the $\ECF$-translation of $(\exists \Theta)\SCF(\Theta)$ is equivalent to $\WKL$, just like for $(\exists \Omega)\MUC(\Omega)$.    
\end{proof}
We now discuss the connection of $\Theta$-functionals to the `classical' fan functional and Nonstandard Analysis.  
\begin{rem}\label{homaar}\rm 
First of all, the first part of Theorem \ref{kinkel} was first proved \emph{indirectly} in \cite{samGH}*{\S3} by applying Theorem~\ref{consresultcor} to the normal form of $\NUC\di \STP$, where 
\be\tag{$\STP$}
(\forall f\leq_{1}1)(\exists^{\st}g^{1}\leq_{1}1)(f\approx_{1}g)
\ee
\be\tag{$\NUC$}
(\forall^{\st}Y^{2})(\forall f^{1}, g^{1}\leq_{1}1)(f\approx_{1}g\di Y(f)=_{0}Y(g)), 
\ee
Note that $\NUC$ expresses that every type two functional is nonstandard uniformly continuous on Cantor space, akin to Brouwer's continuity theorem (\cite{brouw}), 
while $\STP$ expresses the nonstandard compactness of Cantor space as in \emph{Robinson's theorem} (see \cite{loeb1}*{p.\ 42}).   The implication $\NUC\di \STP$ is also proved in Theorem \ref{hirko} below, as it is needed for some related results.  As will become clear in Theorem~\ref{lapdog}, the normal form for $\STP$ gives rise to $\Theta$-functionals, while the normal form of $\NUC$ gives rise to the intuitionistic fan functional $\Omega$.  

\smallskip
Secondly, the `classical' fan functional $\Phi^{3}$ as in $\FF(\Phi)$ below, is obtained from the intuitionistic one by restricting the variable `$Y^{2}$' in $\MUC(\Omega)$ as `$Y^{2}\in C$', where the latter\footnote{Below, we also use `$C$' to denote Cantor space, but no confusion will arise between `$Y^{2}\in C$' and `$f^{1}\in C$' due to the different types.  
} formula expresses continuity as follows:
\be\label{ncont}
Y^{2}\in C\equiv (\forall f^{1})(\exists N^{0})(\forall g^{1})\big[~\overline{f}N=\overline{g}N\di Y(f)=Y(g)\big].  
\ee
\be\tag{$\FF(\Phi)$}
(\forall Y^{2}\in C) (\forall f, g\leq_{1}1)(\overline{f}\Phi(Y)=\overline{g}\Phi(Y)\notag \di Y(f)=Y(g)),   
\ee
By combining \cite{kohlenbach4}*{Prop.\ 4.4 and 4.7}, the `arithmetical comprehension' functional $\exists^{2}$ (also defined in Section \ref{tokie}) can compute (Kleene S1-S9) the classical fan functional, while the proof of Theorem \ref{kinkel} implies that the special fan functional \emph{restricted to $Y^{2}\in C$} can be computed from the classical fan functional.      
\end{rem}
By the previous remark, $\Theta$-functionals can be viewed as a generalisation of (a version of) the classical fan functional to \emph{discontinuous} functionals.  
Such a generalisation is natural in our opinion, as it is well-known that e.g.\ \emph{restricting} oneself to recursive reals and functions, as in the Russian school of recursive mathematics, yields many strange and counter-intuitive results (see \cite{beeson1}*{IV}).  In particular, since discontinuous functions are studied in mainstream mathematics since Riemann's \emph{Habilation} (\cite{kleine}*{p.~115}), it is reasonable to study the generalisations of known functionals to discontinuous inputs (assuming this is well-defined).  
Furthermore, $\Theta$-functionals can be viewed as a version of the classical fan functional \emph{with nonstandard continuity} instead of the epsilon-delta variety by Section \ref{kokkon}.  

\smallskip

In light of the previous observations regarding the classical and intuitionistic fan functionals, special fan functionals appear to be a rather weak objects.     
Looks can be deceiving, as we establish in Theorem \ref{import} that \textbf{no} type two functional can (Kleene S1-S9) compute a special fan functional.   
This includes the \emph{Suslin functional} which corresponds to the strongest Big Five system $\FIVE$ of RM.
Furthermore, the combination of a $\Theta$-functional and $\exists^{2}$, i.e.\ higher-order $\ACA_{0}$ to be introduced in Section \ref{tokie}, turns out to be quite strong, as shown in Sections \ref{hakkenoverdesloot} and~\ref{weako}.

\subsection{The special fan functional and comprehension functionals}\label{tokie}
We study the relationship between special fan functionals and comprehension functionals.  
In particular, we show that the former cannot be computed by the following comprehension functional (or \emph{any} type two functional):
\be\tag{$\exists^{2}$}
(\exists \varphi^{2})(\forall f^{1})\big[(\exists n)(f(n)=0)\asa \varphi(f)=0  \big].
\ee
where we follow the notation from \cite{kohlenbach2}.  To simplify some of the below theorems we reserve `$\exists^{2}$' for the unique functional $\varphi^{2}$ as in $(\exists^{2})$.
Furthermore, we make our notion of `computability' precise as follows.  
\begin{enumerate}
\item[(I)] We adopt $\ZFC$ set theory as the official metatheory for all results, unless explicitly stated otherwise.  
\item[(II)] We adopt Kleene's notion of \emph{higher-order computation} as given by his nine clauses S1-S9 (see \cites{longmann, kleene1}) as our official notion of `computable'.
\end{enumerate}
We assume basic familiarity with computability theory, but introduce aspects of higher-order computability theory as we need them.
We shall often use set theoretic notation when not explicitly working in $\EPA^{\omega}$.    
With these conventions in place, we can prove the following theorem.
\begin{thm}\label{import}
There is no functional $\Theta^3$ satisfying $\SCF(\Theta)$ computable in $\exists^2$.
\end{thm}
\begin{proof}
Assume that $\Theta$ satisfying $\SCF(\Theta)$ is computable in $\exists^{2}$. 
Let $h^{2}$ be any partial functional computable in $\exists^{2}$ which is also total on the class of hyperarithmetical functions; let $g^{2}$ be any total extension of $h$.
By assumption, $\Theta$ applied to $g$ will yield a hyperarithmetical finite sequence $\Theta(g)$.

\smallskip

We now define a particular $h^{2}_{0}$ using \emph{Gandy selection} (\cite{supergandy} and \cite{longmann}*{Theorem~5.4.5}): let $e_{0}$ be the `least' number $e$ such that $e$ is an index for $\alpha$ as a hyperarithmetical function in some fixed canonical indexing of the 
hyperarithmetical sets. By `least' we mean `of minimal ordinal rank', and then of minimal numerical value among those.
Define $h_{0}(\alpha) = e_{0} + 2$ for the aforementioned $e_{0}$ and let $g_{0}$ be a total extension of $h_{0}$.  
Then $\Theta(g_{0})$ will consist of a finite list $\langle\alpha_1 , . . . , \alpha_k\rangle$ of hyperarithmetical functions, and the union of the neighbourhoods determined by the $\overline{ \alpha_i}(g(\alpha_i))$ is not of measure 1. Thus these neighbourhoods cannot cover Cantor space. This contradicts the assumption on $\Theta$.
\end{proof}
 The previous argument is a modification of the proof of the non-computability of the fan functional originally to be found in \cite{gandymahat}. In a letter to Kreisel around 1960 (exact year unknown), Gandy gave a measure-theoretic argument even closer to the one presented here.
\begin{cor}\label{import2}
Let $\varphi^{2}$ be any type two functional.  There is no functional $\Theta^3$ as in $\SCF(\Theta)$ computable in $\varphi$.   
\end{cor}
\begin{proof}
The proof of Theorem \ref{import} relativises to type $2$ functionals computing $\exists^2$.
\end{proof}
We now list some well-known type two functionals which will also be used below.  
\emph{Feferman's search operator} as in $(\mu^{2})$ (see e.g.\ \cite{avi2}*{\S8}) is equivalent to $(\exists^{2})$ over Kohlenbach's system $\RCAo$ by \cite{kooltje}*{\S3}:
\be\tag{$\mu^{2}$}
(\exists \mu^{2})\big[(\forall f^{1})\big( (\exists n^{0})(f(n)=0)\di f(\mu(f))=0   \big)  \big],
\ee
and gives rise to higher-order $\ACA_{0}$.  The \emph{Suslin functional} $(S^{2})$ and the related $(\mu_{1})$ (see \cite{avi2}*{\S8.4.1}, \cite{kohlenbach2}*{\S1}, and \cite{yamayamaharehare}*{\S3}) give rise to higher-order $\FIVE$: 
\be\tag{$S^{2}$}
(\exists S^{2})(\forall f^{1})\big[  (\exists g^{1})(\forall x^{0})(f(\overline{g}n)=0)\asa S(f)=0  \big].
\ee
\be\tag{$\mu_{1}$}
(\exists \mu_{1}^{1\di 1})(\forall f^{1})\big[  (\exists g^{1})(\forall x^{0})(f(\overline{g}n)=0)\di (\forall x^{0})(f(\overline{\mu_{1}(f)}n)=0)  \big].
\ee
On the other hand, full second-order arithmetic as given by $(\exists^3)$ suffices to compute special fan functionals, as we show in Theorem \ref{import3} just below.  
\be\tag{$\exists^3$}\label{hah}
(\exists \xi^{3})(\forall Y^{2})\big[  (\exists f^{1})(Y(f)=0)\asa \xi(Y)=0  \big].
\ee
Similar to the case for $\exists^{2}$, we reserve `$\exists^3$' for the unique functional $\xi^{3}$ from $(\exists^3)$.  
We do the same for other functionals, like $\mu^{2}, \mu_{1}^{2}, S^{2}, \dots$ introduced above.  
\begin{thm}\label{import3}
A functional $\Theta^{3}$ as in $\SCF(\Theta)$ can be computed from $\exists^3$.  
\end{thm}
\begin{proof}
We first prove the existence of a functional $\Theta^{3}$ such that $\SCF(\Theta)$ in $\textsf{ZF}$, i.e.\ classical set theory without the axiom of choice.  
We then show how the construction can be realised as an algorithm relative to $\exists^3$.  
 
\smallskip 

First of all, we introduce some definitions.  Let $C$ be Cantor space $2^\N$ with the lexicographical ordering. If $\sigma$ is a finite binary sequence, we let $C_\sigma$ be the set of binary extensions of $\sigma$ in $C$.  We let $f^{1},g^{1}$ with indices vary over $C$ and we let $\alpha^{1}$, $\beta^{1}$ etc. vary over the countable ordinals. We let $h^{2}$ be a fixed total functional of type two, and our aim is to define $\Theta(h)$.
In particular, by recursion on $\alpha$ we will define an increasing sequence $\{f_\alpha\}_{\alpha < \aleph_1}$ from $C$.
 We put $f_0 := \lambda x.0$ and  
 \[\textstyle
 I(\alpha) := \bigcup_{\beta \leq \alpha}C_{\overline{ f_\beta}h(f_\beta)} \textup{ and } I(< \alpha) := \bigcup_{\beta < \alpha}C_{\overline{ f_\beta}h(f_\beta)}.
 \]
Secondly, consider $\alpha > 0$ and proceed as follows: 
\begin{enumerate}
\item[(I)] If $\lambda x.1 \in I( < \alpha)$, let $f_\alpha = f_\beta$ for the first $\beta$ such that $\lambda x.1 \in C_{\overline{f_\beta}h(f_\beta)}$.
\item[(II)] If not, let $f_\alpha$ be the least element not in $I(< \alpha)$.
\end{enumerate}
By construction, the sequence of $f_\alpha$'s will be strictly increasing until we capture $\lambda x.1$, which thus must happen after a countable number $\alpha^h$ of steps.
 Clearly, the least $\alpha$ such that $f \in I(\alpha)$ must be a successor ordinal for each $f$. Thus, let $\alpha_0 = \alpha^h$ be this ordinal for $f = \lambda x.1$, and let $g_0$ be the greatest strict lower bound of $C_{\overline{f_{\alpha_{0}}}h(f_{{\alpha_{0}}})}.$
 Let $\alpha_1$ be this ordinal for $f = g_0$ and let $g_1$ be the greatest strict lower bound of $C_{\overline{f_{\alpha_1}}h(f_{{\alpha_{1}}})}.$
Continue this process, defining a decreasing sequence $\alpha_0,\alpha_1, \dots$  until $\lambda x.0$ is captured, and we have a finite cover of $C$ of 
neighbourhoods of the form $C_{\overline{f_{\alpha_i}}h(f_{\alpha_i})}$ for $i \leq n$  for some $n$.
 We then define $\Theta(h)$ as the 
 finite sequence $\{f_{\alpha_i}: i  \leq n\}$.
 
 \smallskip 
 
Now observe that $\{\overline{f_{\alpha}}h(f_{\alpha}) : \alpha \leq \alpha^h\}$ is definable as the closure set of a non-monotonic arithmetical inductive definition relative to $h$, so this set will have complexity $\Delta^1_2$ relative to $h$. The extraction of $\Theta(h)$ is arithmetical in this set, so the graph of $\Theta$ is $\Delta^1_2$, and $\Theta$ is computable in $\exists^3$.  
A finer analysis is in Theorem~\ref{mooi}.  
\end{proof}
As it happens, Borel's construction from \cite{opborrelen}*{p.~52} can be applied to our notion of canonical cover, yielding a $\Theta$-functional in the same way as the previous proof. We will refer to the one constructed in the proof of Theorem \ref{import3} as \emph{Borel's $\Theta$}.
Moreover, one needs far less that $\exists^3$ to capture the construction from the proof, but it may be difficult to isolate a weaker `nice' functional in which the special fan functional is computable.  Furthermore, we can refine the previous result to `computation via a term of G\"odel's $T$' if we allow Feferman's mu operator as an additional parameter.  
Let $\SC(\xi)$ be $(\exists^3)$ without the leading existential quantifier.  We refer to \cite{barwise}*{C.7} for an introduction to inductive definitions, while the connection between the latter and $\Theta$-functionals is investigated in \cite{dagcie18}.
\begin{thm}\label{mooi}
There is a term $t^{(2\times 3)\di 3}$ of G\"odel's $T$ such that 
\begin{align}\label{pate}
(\forall  \mu^{2},\xi^{3})&\big[ [\MU(\mu) \wedge \SC(\xi) ] \di  \SCF(t(\mu, \xi))\big],  
\end{align}
and \eqref{pate} is provable in $\EPA^{\omega*}+\X$, where $\X$ expresses that sets may be defined via non-monotonic inductive definitions, and that such sets are $\Delta^1_2$ in the parameters.
\end{thm}
\begin{proof}
First of all, let $\Theta$ be as constructed in the proof of Theorem \ref{import3} and recall that $C$ denotes Cantor space.  As observed in the proof of Theorem \ref{import3}, the graph of $\Theta$ is $\Sigma^1_2$, i.e.\ the formula $\Theta(h^{2}) = 
\langle f_1^{1} , . . . , f^{1}_k\rangle
 $ is equivalent to a $\Sigma_{2}^{1}$-formula with parameters as shown.
Assuming this claim, there is a primitive recursive predicate $S_0$ such that
\[
\Theta(h^{2}) = \langle f_1^{1} , . . . , f^{1}_k\rangle \asa  (\exists g^{1})( \forall z^{1})(\exists n^{0}) S_0\big(h,g,z,n, \langle f_1^{1} , . . . , f^{1}_k\rangle)
\]
and a primitive recursive predicate $S$ such that
\be\label{karju}
\Theta(h)(i) = j \asa (\exists g^{1})( \forall z^{1})( \exists r^1) S(h,i,j,g,z,r(0)),
\ee
where $\Theta(h)(i)$ refers to $f_i$ in the output.
Hence, there is a term $t$ in G\"odel's $T$ which agrees with the characteristic function of $S$.
The exact form of $S$ will depend on how finite sequences are coded, and we need access to the length $k$ of the sequence of functions $\langle f_1^{1} , . . . , f^{1}_k\rangle$ somehow. For this, Feferman's mu-operator is needed, since G\"odel's $T$ only provides bounded search. 

\smallskip

Secondly, we eliminate all quantifiers in \eqref{karju} via $\exists^3$ and obtain a term $t_1$ with parameter $\exists^3$ defining the characteristic function of the right-hand side of \eqref{karju}.  
From this, we use Feferman's mu to extract the values $\Theta(h)(i)$ for $i=1,\ldots , k$.  

\smallskip

Thirdly, in order to prove that the term $t_{1}$ has the desired property, we need axioms proving the totality of $\Theta$ as defined via the process in the proof of Theorem~\ref{import3}.
To this end, if $A$ is a finite set of binary sequences, we put $\Gamma(A) := \{\overline{f}h(f)\}$ where $f^{1}$ is the least binary sequence not covered by $\bigcup_{s \in A}C_s$, if such exists. Otherwise, we put $\Gamma(A) := \emptyset$. Note that $\Gamma$ is a non-monotonic inductive arithmetical operator, and we let $\Gamma^\infty$ be its closure. 

\smallskip

With this definition, $\Gamma^{\infty}$ is a well-ordered set (for the lexicographical ordering) of binary sequences, and such that the corresponding neighbourhoods cover $C$. Given $s = \overline f_\alpha h(f_\alpha) \in \Gamma^\infty$, we can recover $f_\alpha$ as the least function not covered by all $C_t$ for $t < s$ and $t \in \Gamma^{\infty}$. In this way, $\Theta(h)$ is arithmetical in $\Gamma^\infty$ uniformly in $h$.
The only `non-trivial' axiom beyond arithmetical comprehension needed to verify the correctness of this construction is an axiom of inductive definability.
\end{proof}
We finish this section with a note on the use of the intuitionistic mathematics in the formalisation of mathematics in proof assistants.  
\begin{rem}\rm
The proof assistant \emph{Nuprl} is based on Martin-L\"of type theory (\cites{loefafsteken, NUPE}).  
To expedite the laborious process of formalising mathematics, some proofs in Nuprl make use of axioms of Brouwer's intuitionistic mathematics (see e.g.\ \cite{rabideonzin}).  
The latter can have innocent looking classical consequences (like the existence of a special fan functional) which however 
have tremendous computational hardness.   
\end{rem}
\subsection{A weak version of the special fan functional}\label{weak}
We introduce the class of \emph{weak fan functionals} which are in general strictly weaker than special fan functionals.  
As will become clear in Section \ref{hakkenoverdesloot}, weak fan functionals are not just `more of the same' but occupy an important place relative to the special fan functionals.  

\smallskip

Intuitively speaking, where $\Theta(G)$ provides a finite sub-cover for the canonical cover $\cup_{g\in C}[\overline{g}G(g)]$, 
if $\Lambda(G, k)=\langle f_{0}, \dots, f_{m}\rangle$, then the associated finite sub-cover has measure at least $1-\frac{1}{2^{k}}$, i.e.\ as follows: 
\be\label{forp}\textstyle
\m\big(\cup_{i\leq m}[\overline{f_{i}}G(f_{i})] \big)\geq 1-\frac{1}{2^{k}},
\ee
where $\m$ is the usual coin-toss measure on $2^{\N}$.
It is straightforward, but cumbersome, to formally express \eqref{forp} in our formal language.
\bdefi[Weak fan functionals]\label{wdodier}
$\WFF(\Lambda)$ is as follows for $\Lambda^{(2\times 0)\di 1^{*}}$:
\be\textstyle\label{stylez}
(\forall G^{2}, k^{0})\big(\m\big( \cup_{g\in \Lambda(G, k)} [\overline{g}G(g)]\big) \geq 1-\frac{1}{2^{k}}\big).
\ee
Any functional $\Lambda$ satisfying $\WFF(\Lambda)$ is referred to as a \emph{weak fan functional}.
\edefi
Weak fan functionals are not `literally' realisers of theorems from the literature, but these functionals do capture the core complexity of several theorems concerning measure-theoretic approximations, like the Vitali Covering Theorem (\cite{vitaliorg}). This is investigated further in \cite{dagsamVI}.
As it happens, weak fan functionals also arise from \emph{nonstandard compactness}, as discussed in Sections \ref{pampson}.  

\smallskip

In light of the above definitions, there is a (trivial) term of G\"odel's $T$ computing a weak fan functional in terms of a special one.  
We also have the following theorem.
\begin{thm}\label{import22}
There is no functional $\Lambda^{3}$ satisfying $\WCF(\Lambda)$ which is computable in $\exists^{2}$ \(or any type two functional\).   
\end{thm}
\begin{proof}
Analogous to the proof of Theorem \ref{import}.
\end{proof}
As noted above, $\WWKL$ is strictly weaker than $\WKL$, and this is reflected in the following computability result, which is a consequence of Corollary 3.14 and Theorem 3.31 of \cite{dagsamII}.
\begin{thm}\label{import24}
There exists a functional $\Lambda_{1}$ satisfying $\WCF(\Lambda_{1})$ such that no $\Theta$ satisfying $\SCF(\Theta)$ is computable in $\Lambda_{1}$ and $\exists^{2}$.
\end{thm}
Finally, $\Lambda$-functionals relate to $\WWKL$ in the same way $\Theta$-functionals do to $\WKL$.  
\begin{thm}\label{proto38}
$\RCA_{0}^{\omega}+(\exists \Lambda)\WCF(\Lambda)$ is conservative over $\RCA_{0}^{2}+\WWKL$.  
\end{thm}
\begin{proof}
Similar to the proof of Corollary \ref{corhoeker}, one verifies that the $\ECF$-interpretation of $(\exists \Lambda^{3})\WCF(\Lambda)$ follows from $\WWKL$.  
\end{proof}

\section{From computability theory to Nonstandard Analysis}\label{clearflow}
In this section, we use (non)-computability results (some established above) to obtain (negative and postive) results in Nonstandard Analysis.  
By way of a preliminary result and some illustration, we first consider some well-known negative computability theoretic results in Section~\ref{forgo} and derive some negative results in Nonstandard Analysis.  
The main non-implications in Nonstandard Analysis are proved in Section~\ref{fargp}; the computability theoretic results from Section \ref{tokie} are used in an essential way.  
Our other conceptual result is that the RM of Nonstandard Analysis is \emph{fundamentally different} from `usual' RM, in that the 
nonstandard counterparts of the Big Five systems behave quite differently from the originals.  
\subsection{Computability theory and Nonstandard Analysis}\label{forgo}
We show how to translate well-known negative results from higher-order computability theory to negative results in Nonstandard Analysis.  
The former negative results are:
\begin{enumerate}
\item[(a)] There is no \emph{computable} functional that outputs a modulus-of-continuity on input a continuous functional on Baire space.  
\item[(b)] There is no \emph{computable} functional that outputs a modulus-of-uniform-continuity on Cantor space on input a continuous functional on Baire space.  
\end{enumerate}
We now show how these translate to negative results in Nonstandard Analysis.  In particular, the negative results in items (a) and (b) are translated to proofs that 
certain systems of Nonstandard Analysis cannot prove the equivalence of `normal' and `nonstandard' continuity.  

\smallskip
\noindent
First of all, we consider the \emph{modulus-of-continuity functional} $\Psi$ as follows:
\be
(\forall Y^{2}\in C, f^{1}, g^{1})(\overline{f}\Psi(Y, f)=\overline{g}\Psi(Y, f)\notag \di Y(f)=Y(g)). \label{lukl2}\tag{$\textsf{\textup{MPC}}(\Psi)$}
\ee
From $\Psi$ as in $\MPC(\Psi)$, one can define a \emph{discontinuous} type two functional (see \cite{exu} and \cite{beeson1}*{Theorem 19.1}). 
By \cite{kohlenbach2}*{Prop.\ 3.7} and \cite{kohlenbach3}*{\S3}, a discontinuous type two functional can be used to obtain $(\mu^{2})$ inside $\RCAo$.    

\smallskip

By the previous, there is no computable modulus-of-continuity functional.  
As a consequence `normal' continuity \eqref{ncont} does not imply `nonstandard' continuity
\be\label{nscont}
(\forall^{\st} f^{1})(\forall g^{1})(f\approx_{1}g\di Y(f)=_{0}Y(g))
\ee
without extra nonstandard axioms, by the following theorem.  
\begin{thm}\label{simplo}
Let $\varphi$ be internal and such that $\P+\varphi$ is consistent.  
The system $\P+\varphi$ cannot prove that
\be\label{hurg}
(\forall^{\st}Y^{2}\in C)[(\forall^{\st} f^{1})(\forall g^{1})(f\approx_{1}g\di Y(f)=_{0}Y(g))], 
\ee
i.e.\ that all $\epsilon$-$\delta$ continuous functionals are nonstandard continuous \(on Baire space\).  
\end{thm}
\begin{proof}
Let $\varphi$ be as in the theorem and suppose $\P+\varphi$ proves \eqref{hurg}.  The latter is
\[
(\forall^{\st}Y^{2}\in C)(\forall^{\st} f^{1})(\forall g^{1})\big(  (\forall^{\st}k^{0})(\overline{f}k =_{0}\overline{g}k )\di Y(f)=_{0}Y(g)\big)
\]  
with `$\approx_{1}$' resolved.  Pushing outside the standard quantifier involving `$k$', we obtain
\[
(\forall^{\st}Y^{2}\in C)(\forall^{\st} f^{1})(\forall g^{1})(\exists^{\st}k)( \overline{f}k =_{0}\overline{g}k \di Y(f)=_{0}Y(g)).
\]  
Applying idealisation \textsf{I} while bearing in mind Remark \ref{simply}, we obtain:
\be\label{ragga}
(\forall^{\st}Y^{2}\in C)(\forall^{\st} f^{1})(\exists^{\st}N)(\forall g^{1})( \overline{f}N =_{0}\overline{g}N \di Y(f)=_{0}Y(g)).
\ee
Applying Theorem \ref{consresultcor} to `$\P+\varphi\vdash \eqref{ragga}$', we obtain a term $t$ such that 
\[
(\forall Y^{2}\in C, f^{1})(\exists N\in t(Y, f))(\forall g^{1})( \overline{f}N =_{0}\overline{g}N \di Y(f)=_{0}Y(g))
\]
is provable in $\EPA^{\omega*}+\varphi$.  Then $\Psi(Y, f):=\max_{i<|t(Y,f)|}t(Y,f)(i)$ is a computable (even part of G\"odel's ${T}$) modulus-of-continuity functional, a contradiction.
\end{proof}
Note that \eqref{hurg} is provable in $\IST$ by fixing standard $f^{1}$ in \eqref{ncont} and applying the contraposition of \emph{Transfer} to the resulting existential formula.  

\smallskip

Secondly, the fan functional $\Phi^{3}$ as in $\FF(\Phi)$ was introduced by Tait as an example of a functional not computable (Kleene S1-S9; see \cite{gandymahat} or e.g. \cite{longmann}*{\S8}), over the total continuous functionals.  The aforementioned property of the classical fan functional translates to the fact that `normal' continuity does not imply \emph{uniform} nonstandard continuity (on Cantor space), defined as follows:
\be\label{nsucont}
(\forall f^{1}, g^{1}\leq_{1}1)(f\approx_{1}g\di Y(f)=_{0}Y(g)), 
\ee
without the use of nonstandard axioms by the following theorem.  
\begin{thm}
Let $\varphi$ be internal and such that $\P+\varphi$ is consistent.  
The system $\P+\varphi$ cannot prove that
\be\label{hurg2}
(\forall^{\st}Y^{2}\in C)(\forall f^{1} ,g^{1}\leq_{1}1)(f\approx_{1}g\di Y(f)=_{0}Y(g)), 
\ee
i.e.\ $\epsilon$-$\delta$ continuous functionals are nonstandard uniformly cont.\ on Cantor space.
\end{thm}
\begin{proof}
Let $\varphi$ be as in the theorem and suppose $\P+\varphi$ proves \eqref{hurg2}.  Similar to the proof of Theorem \ref{simplo}, \eqref{hurg2} can be brought into the following form:  
\be\label{norma2}
(\forall^{\st}Y^{2}\in C)(\exists^{\st}N^{0})(\forall f^{1} ,g^{1}\leq_{1}1)(\overline{f}N=_{0}\overline{g}N\di Y(f)=_{0}Y(g)).
\ee
Applying Theorem \ref{consresultcor} to `$\P+\varphi\vdash \eqref{norma2}$', we obtain a term $t$ such that  
\[
(\forall Y^{2}\in C)(\exists N\in t(Y))(\forall f^{1},g^{1})( \overline{f}N =_{0}\overline{g}N \di Y(f)=_{0}Y(g))
\]
is provable in $\EPA^{\omega*}+\varphi$.  Then $\Phi(Y):=\max_{i<|t(Y)|}t(Y)(i)$ is a computable (even part of G\"odel's {$T$}) fan functional, a contradiction.
\end{proof}
Note that \eqref{hurg2} is provable in $\IST$ by concluding (inside $\ZFC$) from \eqref{ncont} that $Y^{2}\in C$ is uniformly continuous on Cantor space as follows:
\be\label{refkes}
(\exists N^{0})(\forall f^{1}, g^{1}\leq_{1}1)(\overline{f}N=\overline{g}N\di Y(f)=Y(g)).
\ee
Since $Y^{2}$ in \eqref{hurg2} is standard, we can apply the contraposition of \emph{Transfer} to \eqref{refkes} to obtain uniform nonstandard continuity as in \eqref{nsucont}.  

\smallskip

In conclusion, we have used well-known non-computability results to establish non-implications between the usual and nonstandard definitions of continuity over the system $\P$ extended with any internal sentence.  In other words, certain negative results in computability theory imply that \emph{Transfer} is essential to connect `epsilon-delta' and nonstandard continuity.

\subsection{Reverse Mathematics and Nonstandard Analysis}\label{fargp}
\subsubsection{Introduction\textup{:}\ nonstandard counterparts of the Big Five}
In section \ref{RM}, we observed that the Big Five of RM are linearly ordered as in \eqref{linord}.  
Here, we show that the \emph{nonstandard counterparts} of $\FIVE$, $\ACA_{0}$ on one hand, and of $\WKL_{0}$ and $\WWKL_{0}$ on the other hand, are however \emph{incomparable}.  
Surprisingly, we make essential use of Theorem \ref{import} to establish this result, rather than taking the `usual' model-theoretic\footnote{The fact that the full axiom \emph{Transfer} does not imply the full axiom \emph{Standard Part} is known (over various systems; see \cites{blaaskeswijsmaken, gordon2}), and is established using model-theoretic techniques.} route.  
Thus, the RM of Nonstandard Analysis is \emph{fundamentally different} from `usual' RM, in that the 
nonstandard counterparts of the Big Five systems behave quite differently from the originals. 

\smallskip

Before introducing the aforementioned `nonstandard counterparts', we should clarify what is meant by this expression. 
We stress that there is no deep philosophical meaning to be found in the words `nonstandard counterpart':  in case of $\STP$ and $\LMP$, this is just what these principles are called in the literature: see e.g.\ \cite{pimpson, keisler1, keisler2}.    
Furthermore, term extraction as in Theorem \ref{consresultcor} converts $\paai$ and $\Paai$ into resp.\ $(\mu^{2})$ and $(S^{2})$ (see \cite{sambrouw}*{\S4}), which are higher-order versions of $\ACA_{0}$ and $\FIVE$.  
Thus, the moniker `nonstandard counterpart' seems apt in this case, more so since all the aforementioned nonstandard axioms are natural fragments of the $\IST$-axioms \emph{Transfer} and \emph{Standard Part}.   

\smallskip

We now introduce the nonstandard counterparts of the aforementioned logical systems.  
Recall Nelson's system $\IST$ and the associated fragment $\P$ which were introduced in Section \ref{base}.
The system $\P$ includes Nelson's axiom \emph{Idealisation} (formulated in the language of finite types), but to guarantee a conservative extension of Peano arithmetic, Nelson's axiom \emph{Transfer} must be omitted, while \emph{Standard Part} is weakened to $\HAC_{\INT}$.  Indeed, the fragment of \emph{Transfer} for $\Pi_{1}^{0}$-formulas as follows  
\be\tag{$\paai$}
(\forall^{\st}f^{1})\big[  (\forall^{\st}n)(f(n)\ne0) \di (\forall m)(f(m)\ne0)  \big]
\ee
is the nonstandard counterpart of arithmetical comprehension as in $\ACA_{0}$.  Similar to how one `bootstraps' $\Pi_{1}^{0}$-comprehension to the latter, the system $\P_{0}+\paai$ proves $\varphi\asa \varphi^{\st}$ for any internal arithmetical formula (only involving standard parameters).  
Furthermore, the fragment\footnote{The `bootstrapping' trick for $\paai$ does not work for $\Paai$ (or $\FIVE$) as the latter is restricted to type one objects (like $g^{1}$ in $\Paai$) occurring as `call by value'. } of \emph{Transfer} for $\Pi_{1}^{1}$-formulas as follows  
\be\tag{$\Paai$}
(\forall^{\st}f^{1})\big[ (\exists g^{1})(\forall n^{0})(f(\overline{g}n)=0)\di (\exists^{\st}g^{1})(\forall n^{0})(f(\overline{g}n)=0)\big]
\ee
is the nonstandard counterpart of $\FIVE$.   
The following fragment of \emph{Standard Part} is the nonstandard counterpart of weak K\"onig's lemma (\cite{keisler1, keisler2}):
\be\tag{$\STP$}
(\forall \alpha^{1}\leq_{1}1)(\exists^{\st}\beta^{1}\leq_{1}1)(\alpha\approx_{1}\beta),
\ee  
where $\alpha\approx_{1}\beta$ is short for $(\forall^{\st}n)(\alpha(n)=_{0}\beta(n))$.  
The following fragment of \emph{Standard Part} is the nonstandard counterpart of weak weak K\"onig's lemma (\cite{pimpson}):  
\be\tag{$\LMP$}
(\forall T^{1} \leq_{1}1)\big[ \mu(T)\gg0\di (\exists^{\st} \beta^{1}\leq_{1}1)(\forall^{\st} m^{0})(\overline{\beta}m\in T) \big],
\ee
where `$\mu(T)\gg 0$' is just the formula $(\exists^{\st} k^{0})(\forall^{\st} n^{0})\Big(\frac{\{\sigma \in T: |\sigma|=n    \}}{2^{n}}\geq \frac{1}{k}\Big)$.

\subsubsection{The nonstandard counterpart of $\WKL$}\label{simts}
We study $\STP$, the nonstandard counterpart of $\WKL$.  
While $\FIVE\di \ACA_{0}\di \WKL_{0}$ by \eqref{linord},
we show in Theorem \ref{nogwel} and Corollary \ref{forqu} that the associated \emph{nonstandard counterparts} satisfy $\paai\not\di \STP$ and $\Paai\not\di \STP$ (over $\P$ and extensions).  

\smallskip

As noted above, we shall establish these non-implications in Nonstandard Analysis using Theorem \ref{import}.  
We require the following theorem which provides a normal form for $\STP$ and establishes the latter's relationship with the special fan functional.  
\begin{thm}\label{lapdog}
In $\P_{0}$, $\STP$ is equivalent to either of the following:
\begin{align}\label{frukkklk}
(\forall^{\st}G^{2})(\exists^{\st}w^{1^{*}}\leq_{1^{*}}1, k^{0})\big[(\forall T^{1}\leq_{1}1)\big( & (\forall \alpha^{1} \in w)(\overline{\alpha}G(\alpha)\not\in T)\\
&\di(\forall \beta\leq_{1}1)(\exists i\leq k)(\overline{\beta}i\not\in T)\big) \big],\notag 
\end{align}
\be\tag{N}
(\forall^{\st}G^{2})(\exists^{\st}w^{1^{*}})(\forall f^{1}\leq{1})(\exists g\in w)({f}\in [\overline{g}G(g)]). \label{coredesign}
 \ee
Furthermore, $\P_{0}$ proves $(\exists^{\st}\Theta)\SCF(\Theta)\di \STP$.
\end{thm}
\begin{proof}  
The equivalence $\STP\asa \eqref{coredesign}$ was proved in \cite{dagsamII}*{Theorem 2.6}.
First of all, $\STP$ is easily seen to be equivalent to 
\begin{align}\label{fanns}
(\forall T^{1}\leq_{1}1)\big[(\forall^{\st}n)(\exists \beta^{0})&(|\beta|=n \wedge \beta\in T ) \di (\exists^{\st}\alpha^{1}\leq_{1}1)(\forall^{\st}n^{0})(\overline{\alpha}n\in T)   \big],
\end{align}
and this equivalence may also be found implicitly in \cite{samGH}.  For completeness, we first prove $\STP\asa \eqref{fanns}$.
Assume $\STP$ and apply overspill to $(\forall^{\st}n)(\exists \beta^{0})(|\beta|=n \wedge \beta\in T )$ to obtain $\beta_{0}^{0}\in T$ with nonstandard length $|\beta_{0}|$.  
Now apply $\STP$ to $\beta^{1}:=\beta_{0}*00\dots$ to obtain a \emph{standard} $\alpha^{1}\leq_{1}1$ such that $\alpha\approx_{1}\beta$ and hence $(\forall^{\st}n)(\overline{\alpha}n\in T)$.  
For the reverse direction, let $f^{1}$ be a binary sequence, and define a binary tree $T_{f}$ which contains all initial segments of $f$.  
Now apply \eqref{fanns} for $T=T_{f}$ to obtain $\STP$.    

\smallskip

For \eqref{frukkklk}$\di$\eqref{fanns}, note that \eqref{frukkklk} implies for standard $g^{2}$, there is $k^{0}$ such that: 
\begin{align}\label{frukkklk2}
(\forall T^{1}\leq_{1}1)\big[(\forall^{\st}  \alpha^{1}\leq_{1}1)(\overline{\alpha}g(\alpha)\not\in T),
\di(\forall \beta\leq_{1}1)(\exists i\leq k)(\overline{\beta}i\not\in T) \big], 
\end{align}  
which in turn yields, by bringing all standard quantifiers inside again, that:
\begin{align}\label{frukkklk3}
(\forall T\leq_{1}1) \big[(\exists^{\st}g^{2})(\forall^{\st}\alpha \leq_{1}1)(\overline{\alpha}g(\alpha)\not\in T)\di(\exists^{\st}k)(\forall \beta\leq_{1}1)(\overline{\beta}k\not\in T) \big], 
\end{align}  
To obtain \eqref{fanns} from \eqref{frukkklk3}, apply $\HAC_{\INT}$ to $(\forall^{\st}\alpha^{1}\leq_{1}1)(\exists^{\st}n)(\overline{\alpha}n\not\in T)$ to obtain standard $\Psi^{1\di 0^{*}}$ such that  
$(\forall^{\st}\alpha^{1}\leq_{1}1)(\exists n\in \Psi(\alpha))(\overline{\alpha}n\not\in T)$, and defining $g(\alpha):=\max_{i<|\Psi|}\Psi(\alpha)(i)$ we obtain $g$ as in the antecedent of \eqref{frukkklk3}.  Hence, \eqref{frukkklk3} yields
\be\label{gundark}
(\forall T^{1}\leq_{1}1) \big[(\forall^{\st}\alpha^{1}\leq_{1}1)(\exists^{\st}n)(\overline{\alpha}n\not\in T)\di (\exists^{\st}k)(\forall \beta\leq_{1}1)(\overline{\beta}i\not\in T) \big], 
\ee
which is the contraposition of \eqref{fanns}, using classical logic.  For the implication $\eqref{fanns}  \di \eqref{frukkklk}$, consider the contraposition of \eqref{fanns}, i.e.\ \eqref{gundark}, and note that the latter implies \eqref{frukkklk3}.  Now push all standard quantifiers outside as follows:
\[
(\forall^{\st}g^{2})(\forall T^{1}\leq_{1}1)(\exists^{\st} ( \alpha^{1}\leq_{1}1, ~k^{0}))\big[(\overline{\alpha}g(\alpha)\not\in T)
\di(\forall \beta\leq_{1}1)(\exists i\leq k)(\overline{\beta}i\not\in T) \big], 
\]
and applying idealisation \textsf{I} yields \eqref{frukkklk}.  The equivalence involving the latter also immediately establishes the second part of the theorem.    
\end{proof}
\begin{cor}\label{conske}
The system $\P_{0}+\STP$ is conservative over $\RCA_{0}^{2}+\WKL$. 
\end{cor}    
\begin{proof}
Let $\varphi$ be a sentence in the language of $\RCA_{0}^{2}$.  If $\P_{0}+\STP\vdash \varphi$, then $\P_{0}\vdash (\exists^{\st}\Theta)\SCF(\Theta)\di \varphi$ by the theorem.  Applying Theorem~\ref{consresultcor} to $\P_{0}\vdash (\forall^{\st}\Theta)(\SCF(\Theta)\di \varphi)$ yields 
$\RCA_{0}^{\omega}\vdash (\forall \Theta)(\SCF(\Theta)\di \varphi) $, and Corollary \ref{corhoeker} finishes the proof.  
\end{proof}
In light of the previous theorem, the `nonstandard' provenance of special fan functionals becomes clear.  
Indeed, these were actually discovered during the study of the Gandy-Hyland functional in \cite{samGH}*{\S3-4}, as discussed in Footnote \ref{hitherto}.  

\smallskip

Thirdly, we establish the aforementioned non-implications and related results.  In the case of independence results like in the following theorem, we always implicitly assume the system at hand to be  consistent.  
\begin{thm}\label{nogwel}
The system $\P+\paai$ does not prove $\STP$.  
\end{thm}
\begin{proof}
Suppose $\P+\paai\vdash \STP$ and note that $\paai$ is equivalent to 
\be\label{trakke}
(\forall^{\st}f^{1})(\exists^{\st}n^{0})\big[ (\exists m)f(m)=0 \di (\exists i\leq n)f(i)=0  \big],
\ee
by contraposition.  Then the implication `$\paai\di \STP$' becomes
\be\label{jaj}
(\forall^{\st}f^{1})(\exists^{\st}n^{0})A(f,n)\di (\forall^{\st}g^{2})(\exists^{\st}w^{1^{*}},k^{0})B(g,w,k)
\ee
where $B$ is the formula in square brackets in \eqref{coredesign} and where $A$ is the formula in square brackets in \eqref{trakke}.  
We may strengthen the antecedent of \eqref{jaj} as follows:
\be\label{jaj2}
(\forall^{\st} h^{2})\big[(\forall^{\st}f^{1})A(f,h(f))\di (\forall^{\st}g^{2})(\exists^{\st}w^{1^{*}},k^{0})B(g,w,k)\big], 
\ee  
In turn, we may strengthen the antecedent of \eqref{jaj2} as follows:
\be\label{jaj3}
(\forall^{\st} h^{2})\big[(\forall f^{1})A(f,h(f))\di (\forall^{\st}g^{2})(\exists^{\st}w^{1^{*}},k^{0})B(g,w,k)\big], 
\ee  
Bringing out the standard quantifiers, we obtain
\be\label{jaj4}
(\forall^{\st} h^{2}, g^{2})(\exists^{\st}w^{1^{*}},k^{0})\big[(\forall f^{1})A(f,h(f))\di B(g,w,k)\big], 
\ee  
and applying Theorem \ref{consresultcor} to `$\P\vdash \eqref{jaj4}$', we obtain a term $t$ such that         
\be\label{jaj5}
(\forall h^{2}, g^{2})(\exists w^{1^{*}}, k^{0}\in t(h,g))\big[(\forall f^{1})A(f,h(f))\di B(g,w,k)\big], 
\ee  
is provable in $\textsf{E-PA}^{\omega*}$.  Clearly, the antecedent of \eqref{jaj5} expresses that $h$ is Feferman's search functional $\mu^{2}$.  
Furthermore, it is straightforward to define $\Theta$ as in $\SCF(\Theta)$ in terms of $(\lambda g)t(h,g)$;  However, this implies that a special fan functional is computable in $\mu^{2}$ via a term from G\"odel's ${T}$.  This contradicts Corollary \ref{import2}.   
\end{proof}
In the previous proof, we observed that applying Theorem \ref{consresultcor} results in $\paai$ being converted to Feferman's mu operator, which is a kind of comprehension axiom (with a dash of choice).  
The same holds for other instances of \emph{Transfer}, like in the folllowing corollary.
\begin{cor}\label{forqu}
The system $\P+\Paai$ does not prove $\STP$.  
\end{cor}
\begin{proof}
Follows from Corolllary \ref{import2} in the same way as the theorem. 
Indeed, $\Paai$ has the following normal form:
\[
(\forall^{\st}f^{1})(\exists^{\st}g^{1})\big[ (\exists g^{1})(\forall x^{0})(f(\overline{g}n)=0)\di (\forall x^{0})(f(\overline{g}n)=0)\big], 
\]
and hence applying Theorem~\ref{consresultcor} to `$\P+\Paai\vdash \STP$' yields, in the same way as in the theorem, a term of G\"odel's $T$ converting $\mu_{1}$ to a special fan functional.    
\end{proof}
Similarly, Corollary \ref{import2} yields that \emph{Transfer} limited to $\Pi_{k}^{1}$-formulas cannot imply $\STP$.  
Indeed, the `comprehension functional' for $\Pi_{k}^{1}$-formulas has type two, and hence does not compute any special fan functional by Corollary \ref{import2}. 
Similarly, we can obtain the non-implication `$\P+\Paai+\varphi\not\vdash\STP$' for $\varphi$ \textbf{any} internal sentence (provable in $\ZFC$ and such that the former system is consistent).  Finally, the same holds for certain \emph{external} sentences, like $\WKL^{\st}$ and $\ATR^{\st}$, as long as they follow from $\Paai$ (or \emph{Transfer} limited to $\Pi_{k}^{1}$-formulas).

\smallskip

Finally, we derive $\STP$ using the following versions of \emph{Transfer}:
\be\tag{$\SOT$}
(\forall^{\st} Y^{2})\big[ (\exists f^{1})(Y(f)=0)\di  (\exists^{\st} f^{1})(Y(f)=0)  \big],
\ee
\be\tag{$\TOT$}
(\forall^{\st} Z^{3})\big[ (\exists Y^{2})(Z(Y)=0)\di  (\exists^{\st} Y^{2})(Z(Y)=0)  \big].
\ee
Recall the axiom $\X$ from Theorem \ref{mooi}; we obtain the following theorem.
\begin{thm}
The system $\P+\X+\TOT$ proves $\STP$.  
\end{thm}
\begin{proof}
By Theorem \ref{mooi}, \eqref{pate} is also provable in $\P+\X$.  For standard $\mu^{2}$ and $\exists^3$, the term $t$ provides standard output by Definition \ref{debs}, i.e.\ $\P+\X$ proves
\begin{align}\label{pate2}
(\forall^{\st} \mu^{2}, \xi^{3})&\big[ [\MU(\mu) \wedge \SC(\xi) ] \di (\exists^{\st}\Theta) \SCF(\Theta)\big].  
\end{align}
The theorem now follows from Theorem \ref{lapdog} and $\TOT\di \SOT \di \paai$, $\SOT\di (\exists^{\st}\mu^{2})\MU(\mu)$ and $\TOT\di (\exists^{\st}\xi^{2})\SOC(\xi)$, which are readily proved.   
\end{proof}
Finally, we discuss the connection between standardness and computability.  
\begin{rem}[Standardness and computability]\label{bestio}\rm
The previous proof hinges on the basic axioms of $\P$ from Definition \ref{debs}, 
which imply that the standard functionals in $\P$ are closed under `computability via a term from G\"odel's $T$'.  
It is then a natural question whether the standard functionals (resp.\ functions) in $\P$ are closed under (resp.\ Turing) computability?  
As it turns, out, the answer depends on the presence of \emph{Transfer}: 
in case of Turing computability, one readily proves that $\paai$ is equivalent to the aforementioned closure property, while one seems to require prohibitively strong fragments of \emph{Transfer} to guarantee this property for functionals of higher type.  
Thus, `computability via a term from G\"odel's $T$' produces results in $\P$ (and vice versa by Theorem \ref{consresultcor}), but `S1-S9 computability' only seems to produce results in extremely strong extensions of $\P$.  
\end{rem}
The previous remark explains why we insisted on obtaining Theorem \ref{mooi}, and the term from G\"odels $T$ therein in particular.  
In conclusion, we have shown that the computability theoretic results from Section \ref{tokie} give rise to (non-)implications in the RM of Nonstandard Analysis.  
In particular, quite strong fragments of \emph{Transfer} do not imply the weak version of \emph{Standard Part} as in $\STP$.  As a bonus, these results imply that the RM of Nonstandard Analysis is quite different from `vanilla' RM, as will be further explored in the following sections.   

\subsubsection{The nonstandard counterpart of $\WWKL$}\label{pampson}
We study $\LMP$, the nonstandard counterpart of $\WWKL$.  
While $\FIVE\di \ACA_{0}\di \WWKL_{0}$ by \eqref{linord},
we show in Theorem \ref{nogwelke} that the associated \emph{nonstandard counterparts} satisfy $\paai\not\di \LMP$ and $\Paai\not\di \LMP$, all over the system $\P$.  

\smallskip
As noted above, we shall establish these non-implications in Nonstandard Analysis using Theorem \ref{import}.  
We require the following theorem which provides a normal form for $\LMP$ and establishes the latter's relationship with the weak fan functional. 
\begin{thm}\label{lapdoc}
In $\P_{0}$, the principle $\LMP$ is equivalent to either of the following:
\begin{align}\label{w2}\textstyle
(\forall^{\st}G^{2},k^{0})(\exists^{\st}w^{1^{*}}&\leq_{1^{*}}1, n^{0})\\
&\textstyle(\forall T\leq_{1}1)\big[(\forall \alpha\in w)(\overline{\alpha}G(\alpha)\not\in T)\di \frac{|\{\sigma \in T: |\sigma|=n    \}|}{2^{n}}\leq\frac1k\big].\notag
\end{align}  
\be\label{tinkle}\textstyle
(\forall^{\st} G^{2}, k^{0})(\exists^{\st} w^{1^{*}})\big(\m\big( \cup_{g\in w} [\overline{g}G(g)]\big) \geq 1-\frac{1}{2^{k}}\big).
\ee
Furthermore, $\P_{0}$ proves $(\exists^{\st} \Lambda)\WCF(\Lambda)\di \LMP$.  
\end{thm}
\begin{proof}
Analogous to the proof of Theorem \ref{lapdog}. 
\end{proof}
A system is called \emph{robust} (see \cite{montahue}*{p.\ 432}) in Reverse Mathematics if it is equivalent to small perturbations of itself.  
It is an easy exercise to verify that $\STP \asa \LMP'$, where the latter is $\LMP$ with `$\mu(T)>_{\R}0$' rather than $\mu(T)\gg 0$.  On the other hand, $\STP$ is equivalent to \eqref{fanns} with 
the `st' in the antecedent removed.  Hence, $\STP$ seems to be robust, while $\LMP$ is not.  
Nonetheless, we have the following version of Corollary \ref{forqu} for $\LMP$.  
\begin{thm}\label{nogwelke}
The system $\P+\Paai$ does not prove $\LMP$.  
\end{thm}
\begin{proof}
Analogous to the proof of Theorem \ref{nogwel} by Theorem \ref{lapdoc}.
\end{proof}
The following Theorem establishes the nonstandard version of the non-implication $\WWKL\not\di \WKL$, which was first proved in \cite{yussie}.  
 \begin{thm}
The system $\P_{0}+\LMP$ does not prove $\STP$.  
\end{thm}
\begin{proof}
We proceed similar to Theorem \ref{nogwel}.   Suppose $\P_{0}+\LMP\vdash \STP$;  in the same way as for the aforementioned theorem,
we obtain some term $t$ such that $\RCAo$ proves $(\forall \Lambda)(\WCF(\Lambda)\di \SCF(t(\Lambda)))$.  
In particular $\RCAo+(\exists \Lambda)\WCF(\Lambda)$ proves $(\exists \Theta)\SCF(\Theta)$.
Since $(\exists \Theta)\SCF(\Theta)\di \WKL$ over $\RCAo$, we have that $\RCAo+(\exists \Lambda)\WCF(\Lambda)$ proves $\WKL$, contradicting Corollary \ref{proto38}.  
We could obtain a similar contradiction from Theorem \ref{import24}.  
\end{proof}
The following theorem generalises the previous result.
\begin{thm}\label{firstamen}
The system $\P+\paai+\LMP$ does not prove $\STP$.
\end{thm}
\begin{proof}
Follows from Theorem \ref{import24} in the same way as Theorem \ref{nogwel} follows from Corollary \ref{import2}.
In particular, suppose $\P+\paai+\LMP$ does prove $\STP$ and note that following the proof of Theorem \ref{nogwel}, we obtain a term $t$ of G\"odel' $T$ computing the special fan functional in terms of $\exists^{2}$ and a weak fan functional.  However, this contradicts Theorem \ref{import24}.  
An alternative proof is given in Corollary \ref{dagdag} below.   
\end{proof}
The following corollary, a weak version of Theorem \ref{import24}, is now straightforward.
 \begin{cor}
 Let $\varphi$ in the language of $\textsf{\textup{E-PA}}^{\omega*}$ be such that the latter plus $\varphi$ is consistent.  
For any term $t$ of G\"odel's $T$, $\textsf{\textup{E-PA}}^{\omega*}+\varphi$ does not prove 
\[
(\forall \Lambda^{3}, \mu^{2})\big([\WCF(\Lambda)\wedge \MU(\mu)]\di \SCF(t(\Lambda))\big).
\]  
\end{cor}
We will sharpen the previous corollary in Section \ref{hakkenoverdesloot} via a detailed analysis of the computational power of the special and weak fan functionals.
\section{A more refined analysis of weak and special fan functionals}\label{hakkenoverdesloot}
In this section, we show that (certain) weak fan functionals are indeed computationally weaker than (all) special fan functionals, as follows. 
Intuitively speaking, we show that $\Theta$-functionals always can escape a certain well-known computational class, called the \emph{hyperarithmetical} functionals, while there is a $\Lambda$-functional that does not escape this class. 
\subsection{Introduction}
In the previous sections, we have established a number of striking properties of the special and weak fan functionals and $\exists^{2}$.  
This section is devoted to a detailed analysis of the computational power of the aforementioned functionals and their combinations. For the sake of readability, we will use capital letters from the Latin alphabet to denote objects of type 2.

\smallskip

As a result of our refined analysis, certain weak fan functionals will be established as being \emph{weaker} than special ones \emph{in the following concrete way}:  there exists 
a weak fan functional which provides hyperarithmetical output for hyperarithmetical input, but no such special fan functional exists.  
These results are interesting in their own right, but are also the key to the results from Section \ref{weako}.

\smallskip

We recall the agreed-upon meaning of `computable' (Kleene S1-S9) and metatheory ($\ZFC$) from Section \ref{tokie}.  
In this section, we will rely heavily on the classical theory for the hyperarithmetical, $\Pi^1_1$, and $\Sigma^1_1$-sets, and on the computability theory of $\exists^{2}$. We do not give original references to each result we make use of, but refer to \cite{Sacks.high} for an introduction to the field.  

\smallskip

Section \ref{noteje} is devoted to the proof of Theorem \ref{refine}, which has useful corollaries.  In \cite{dagsamII} it is proved that any special fan functional $\Theta$ computes a realiser  for arithmetical transfinite recursion, which is sufficient for proving Corollary \ref{cor.5..5}. Theorem \ref{refine} was proved prior to this result from  \cite{dagsamII}, and even though some of the consequences can be proved differently, the construction in the proof of Theorem \ref{refine} may be of independent interest.
\begin{thm}\label{refine}
There is a total functional $F:2^\N \rightarrow \N$ computable in $\exists^2$ such that the set of neighbourhoods $C_{\overline{f}F(f)}$, where $f$ varies over all binary hyperarithmetical functions, is not a cover of $2^\N$.
\end{thm}
Recall the intuitive description of $\Theta$-functionals right below Definition~\ref{dodier} and recall that functionals computable in $\exists^{2}$ only produce hyperarithmetical functions; we have the following immediate corollary.  
\begin{cor} \label{cor.5..5}
For any $\Theta$ as in $\SCF(\Theta)$, there are more functions of type one computable in $\Theta$ and $\exists^2$ than just in $\exists^2$. 
\end{cor}
For a further discussions of Theorem \ref{refine}, we refer to Section \ref{thetatje} where we also prove that the combination of Borel’s $\Theta$ and $(\exists^{2})$ computes the Suslin functional. 

\smallskip

Finally, in Section \ref{kood}, we will construct a particular functional $\Lambda_{0}$ such that $ \WCF(\Lambda_{0})$ and which yields hyperarithmetical output for hyperarithmetical input.

\subsection{The proof of Theorem \ref{refine}}\label{noteje}
We prove Theorem \ref{refine} in Section \ref{noteje3}, but first introduce some necessary notations and preliminaries in Section \ref{noteje2}.  
\subsubsection{Notation and preliminaries}\label{noteje2}
To save space, some claims are described as `Fact'; proofs can be found in text-book level literature like \cite{Sacks.high, Rogers}.  
We make use of the following `standard' definitions which can be found in any textbook.  
\bdefi[Basic Notations]~
\begin{enumerate}
\item  Let $\phi_e$ denote the partial computable function with index $e$ as obtained from the Kleene $T$-predicate.
\item Similarly, $\phi_{e}^A$ denotes partial function number $e$ with oracle $A\subset \N$. 
\item We let ${\mathcal K}^A$ be the \emph{jump} of $A$, i.e.\ the set
$\{e : \phi_e^A(e) \!\!\downarrow\}$.  
\item 
Kleene's set `$O$' with the partial ordering `$\prec$' is the minimal $\langle O,\prec\rangle$ s.t.\
\begin{enumerate}
\item $0 \in O$ and $a \in O \Rightarrow 2^a \in O \wedge a \prec 2^a$,
\item if $\phi_e(n) \in O$ for all $n$ and $\phi_e(n) \prec \phi_e(n+1)$ for all $n$, then $3 \cdot 5^e \in O$ and $\phi_e(n) \prec 3 \cdot 5^e$ for all $n$,
\item the partial ordering $\prec$ is partial.
\end{enumerate}
\end{enumerate}
\edefi
\begin{fact} There is an arithmetical end-extension $\langle O^+,\prec^+\rangle$ of $\langle O,\prec\rangle$ that is a fixed point of the inductive definition defining $O$ and $\prec$, and such that all initial segments are totally ordered.\end{fact}
We will let `$\{e\}(\exists^2,\vec a) = b$' mean that the computable functional with index $e$ and inputs $\exists^2$ and the number sequence $\vec a$, terminates with value $b$.
\begin{fact}
For hyperarithmetical $A^{1}$, its characteristic function is computable in $\exists^2$.
\end{fact}
\begin{fact}
There is a classically computable, total function $\rho$ such that for all $e$, $\vec a$ and $b$, we have $\rho(e,\vec a , b) \in O \Leftrightarrow \{e\}(\exists^2,\vec a) = b$.
\end{fact}
\begin{defi}{\em Let $b \in O^+$. 
A \emph{$b$-chain} will be a set $\{H_a\}_{a \preceq^+ b}$ such that
\begin{enumerate}
\item[a)] $H_0 = \emptyset$ and if $a = 2^c$ then $H_a = {\mathcal K}^{H_c}$.
\item[b)] If $a = 3\cdot 5^e$,  then $H_a = \{\langle n,m\rangle : m \in H_{\phi_e(n)}\}$.   
\end{enumerate}}\end{defi}
\begin{fact}\label{fact4} We have the following properties of $b$-chains.
\begin{itemize}
\item[a)] If $b \in O^+$ then there is a hyperarithmetical $b$-chain if and only if $b \in O$.
\item[b)] There is a Kleene index $e_0$ such that for all $a \in O$ and $c \in \N$:
\[
\{e_0\}(\exists^2,a,c) = \left \{ \begin{array}{ccc}1&{\rm if}&c \in H_a\\0 & {\rm if}& c \not \in H_a\end{array} \right..
\]
\item[c)] The set of $b$-chains is uniformly arithmetically defined for any $b \in O^+$.
\end{itemize}
\end{fact}
The above facts constitute (partly) the key steps in the proof of the Spector-Gandy theorem (\cite{gandys,spector}, see also  \cite{Sacks.high}*{p.\ 61}).
\begin{rem}[Well-orderings and the hyperarithmetical]\rm For any $b \in O$, there is exactly \emph{one} $b$-chain and the latter is definable using {\em arithmetical transfinite recursion} as formalised in $\ATR_0$. 
One technical challenge in our proof of Theorem \ref{refine} is that there are elements $c\in O^+ \setminus O$ for which there is neither a hyperarithmetical $c$-chain nor a hyperarithmetical descending sequence.  

\smallskip

On the other hand, in the proof of Theorem \ref{rejeji} we exploit the existence of such $c$ to obtain a negative result while the associated Corollary \ref{cor.alt.6.8} yields a `softer' proof of the main theorem of this section.  Still, we find the explicit construction here to be of independent interest.
\end{rem}

\subsubsection{The construction establishing Theorem \ref{refine}}\label{noteje3}
We construct the functional $F$ from Theorem \ref{refine}.  
To this end, let $\alpha^{1}$ be the following partial binary function:
\[
\alpha(e) :=
\begin{cases}
\{e\}(\exists^2,e) & \textup{ if $\{e\}(\exists^2,e) \in \{0,1\}$}\\
\textup{undefined}  &\textup{otherwise}
\end{cases}
\]
Let $X$ be the set of all total binary functions extending $\alpha$.  Hence, $X$ is a non-empty, closed  $\Sigma^1_1$-set with no\footnote{That $X$ contains no hyperarithmetical elements is proved in the same way as one proves that the Kleene-tree has no computable infinite branches, just relativised to computability in $\exists^2$.} hyperarithmetical elements. 
\begin{lem}\label{lemma1.6} If $f \in X$ and $\{e\}(\exists^2, \vec a) \in \{0,1\}$, then we can, uniformly $\mu$-recursive in $f$, find  $\{e\}(\exists^2, \vec a)$. \end{lem}
\begin{proof}
There is a primitive recursive function $\xi$ such that if $\{e\}(\exists^2,\vec a)\!\!\downarrow$ then $$\{\xi(e,\vec a)\}(\exists^2,\xi(e,\vec a )) = \{e\}(\exists^2,\vec a).$$
This is seen by a simple index manipulation using only Kleene's S1-S7.
Then $\{e\}(\exists^2,\vec a) = f(\xi(e,\vec a))$ and we are done.  
\end{proof}
We are now ready to give the proof of Theorem \ref{refine}.
 \begin{proof} 
Given  a binary $f^{1}$ we will look for two sorts of evidence: evidence that $f \in X$ and evidence of the opposite. If we, for each $e$, gather evidence for $f(e)$ being compatible with $\alpha(e)$, our construction will ensure that $f \in X$, and we may put $F(f) = 0$. This is because $f$ is not hyperarithmetical in this case. 

\smallskip

If we, for some $e$, find an indication of $f(e)$ being incompatible with $\alpha(e)$, we will give $F(f)$ a value so large that an alleged incompatibility is manifested for some $x < F(f)$. 
We will see to it that if $f$ is hyperarithmetical (something that cannot be decided, that is the underlying problem) then the alleged incompatibility is a real one.
Asking for compatibility at  $e$ is the same as asking if we have:   $$\neg(\{e\}(\exists^2,e) = 1 - f(e)).$$
This is the same as asking: is $\rho(e,e,1-f(e)) \not \in O?$
If $\rho(e,e,1-f(e)) \not \in O^+$, we have confirmation of the compatibility at $e$, so assume that $\rho(e,e,1-f(e)) \in O^+$.

\smallskip
We now employ the index $e_0$ from Fact \ref{fact4}.b) and the algorithm from Lemma \ref{lemma1.6}. From $f$, compute an alleged $\rho(e,e,1-f(e))$-chain of the form $\{H^f_a\}_{a \preceq \rho(e,e,1-f(e))}$, i.e.\ we let $H_a^f$ be the set with characteristic function $\lambda b.f(\xi(e_0,a,b))$.    Given $e$, there will be three possibilities, and $\exists^{2}$ can decide which one holds:
\begin{enumerate}
 \renewcommand{\theenumi}{\roman{enumi}}
\item $\{H^f_a\}_{a \preceq \rho(e,e,1-f(e))}$ is a proper chain.
\item $\{H^f_a\}_{a \preceq \rho(e,e,1-f(e))}$ is not a chain, and there is no least place where the inductive definition breaks down.
\item $\{H^f_a\}_{a \preceq \rho(e,e,1-f(e))}$ is not a chain, and there is a least place where the induction breaks down.
\end{enumerate}
For each of these possibilities, we will either conclude that we have a confirmation of the compatibility of $f$ with $\alpha$ at $e$, 
or we will find a value $x_e$ such that we \emph{may} let $F(f) = x_e + 1$. The point is that if $f$ is hyperarithmetical, 
then we find some $x_e$, and any choice of $x_e$ will be such that $f$ and $\alpha$ are incompatible at $x_e$.
Thus, no extension of $\overline{ f}F(f)$ will be in $X$ with this choice of $F(f)$.

\smallskip
In case of (i), if $f$ is hyperarithmetical, then the chain is hyperarithmetical; due to Fact \ref{fact4}.a), $\rho(e,e,1-f(e)) \in O$, so $\{e\}((\exists^2),e) = 1-f(e)$. In this case put $x_e = e$.

\smallskip
In case of (ii), we have spotted an arithmetical non-empty subset of the $O^+$-initial segment of $\rho(e,e,1-f(e))$ without least element.  This implies $\rho(e,e,1-f(e)) \not \in O$ and yields a confirmation of the compatibility of $f$ and $\alpha$ at $e$.

\smallskip
This leaves us with case (iii).  Let $a$ be the least element in the initial segment  of $\rho(e,e,1-f(e))$ where the chain constructed from $f$ fails to satisfy the induction. 
This means that if $H$ is the candidate for the chain at $a$ (that we arithmetically define from the corresponding initial segment of the chain), then $H \neq H^f_a$. Viewing $H$ and $H^f_a$ as characteristic functions, there will be a least $b$ such that $H(b) \neq H^f_a(b) = f(\xi(e_0,a,b))$. We let $x_e = \xi(e_0,a,b)$ in this case.

\smallskip
If, in this case, $f$ is hyperarithmetical, we must have that $a \in O$, by Fact \ref{fact4}.a), since there is a proper chain up to $a$. This implies in turn that if $H$ is the set defined above, $H$ is really $H_a$, which is computed from $\exists^2$ by
\[
H_a(b) = \{e_0\}(\exists^2,a,b) = \phi(\xi(e_0,a,b)).
\]
Thus, the least $b$ chosen as above will, in this case, give  a correct witness $x_e = \xi(e_0,a,b)$ to the fact that $f$ is incompatible with $\phi$.  

\smallskip
We can now finalise the definition of $F(f)$ as follows:
\begin{enumerate}
\item[(i)] If we, for all $e$, obtain a confirmation of the compatibility of $f(e)$ and $\alpha(e)$ as above, we let $F(f) = 0$. In this case, $f$ is not hyperarithmetical.
\item[(ii)] Otherwise, let $x$ be minimal such that there is $e$  for which we do not have a confirmation like this by the considerations above and $x = x_e$. 
We let $F(f) = x + 1$.  For hyperarithmetical $f$, $\overline{ f}F(f)$ has no extension in $X$. 
\end{enumerate}
As is easily verified, we never left the arithmetical in our constructions, so $F$ is, with good margin, computable in $\exists^2$. 
The construction ensures that $X$ is disjoint from $\{g : \overline{f}F(f) \subset g\}$ whenever $f$ is hyperarithmetical.  
\end{proof}
\subsection{Computing the Suslin functional from Borel's $\Theta$}\label{thetatje} 
In this section, we show that the Suslin functional is computable in the particular special fan functional called \emph{Borel's $\Theta$}, which was introduced in Section \ref{tokie}.

\smallskip

As to the history of this result, in a preprint version of this paper (see \cite{dagsam.I.p}), we proved that Borel's $\Theta$, when applied to the functional  $F$ constructed in the proof of Theorem \ref{refine}, yields a function with the same Turing degree as a complete $\Pi^1_1$-set; from this we concluded that the Suslin functional is computable in the functional Borel's $\Theta$. 
In \cite{dagcie18}, this fact is used to prove that the closure operator for non-monotone inductive definitions, seen as a functional of type 3, is computable in $\exists^2$ and Borel's $\Theta$. 
Later, we discovered a more transparent proof, showing \emph{directly} that Borel's $\Theta$ can decide if a total ordering is a well-ordering or not, and this argument replaces in this paper the original content of Section~\ref{thetatje} from \cite{dagsam.I.p}.  We warn the reader that due to the rewrite of Section \ref{thetatje}, the numbering in this section has been changed from \cite{dagsam.I.p}.
The numbering in the latter was used when writing e.g.\ \cite{dagsamII}.

\smallskip

First of all, we introduce a decision procedure for well-orderings relative to Borel's $\Theta$, as follows.  
Intuitively, given a total ordering $R$ of $\N$, we can consider the tree $T_R$ of sequences $\langle n_0 , \ldots , n_k\rangle$ that are strictly increasing in the ordering of $\N$ and strictly decreasing in the ordering  $R$. Then $R$ is a well-ordering if and only if $T_R$ is well founded. We may then, informally, use a transfinite top-down, left-to-right search for an infinite branch in $T_R$ in order to decide if $R$ is a well-ordering. 
This intuition can be formalised as follows. 
\begin{thm} 
Let $\Theta_0$ be Borel's $\Theta$. Let $R$ be a binary relation on $\N$. Uniformly in $R$ there is an arithmetical functional $F_R$ such that we can decide, arithmetically in $R$ and $\Theta_0(F_R)$, if $R$ is a well-ordering of $\N$ or not.\end{thm}
\begin{proof}
Since we may arithmetically decide if $R$ is a total ordering or not, we assume that it is, and rename it $<_R$. We let $<_L$ be the lexicographical ordering of $C$. When we evaluate $\Theta_0$ on $F$, we are constructing an $<_L$-increasing sequence $\{f_\gamma\}_{\gamma \leq \alpha}$ where $\alpha$ is a countable ordinal and the following holds:
\begin{itemize}
\item[(i)] The function $f_0$ is constant 0.
\item[(ii)] If $\gamma$ is a limit ordinal, then $f_\gamma = \sup\{f_\beta \mid \beta < \gamma\}$ in the sense of $<_L$.
\item[(iii)] If $\gamma + 1 \leq \alpha$, then $\bigcup\{C_{\bar f_\beta(F(f_\beta))} \mid \beta \leq \gamma\}$   is a proper initial segment of  $C$, and $f_{\gamma + 1}$ is the $<_L$-least element in the complement.
\item[(iv)] The collection $ \bigcup\{C_{\bar f_\beta(F(f_\beta))} \mid \beta \leq \alpha\}$ covers Cantor space $C$.
\end{itemize}
From the cover in item (iv), we extract a finite sub-covering from right to left. In particular, if we hit upon some $f$ such that $F(f) = 0$ in this process, this $f$ will be our $f_\alpha$, and $\Theta_0(F) = \{f_\alpha\}$. Our aim is to construct $F_R$ such that this will be the case whenever $R$ is not a well-ordering, and then $f_\alpha$ will code a $<_R$- descending sequence.

\smallskip

For $f \in C$, define $A_f := \{n \mid f(n) = 0\}$, which is enumerated (in $\N$-increasing order) as $\{m^f_k\}_{k < N^f}$, and where $N^f \in \N \cup \{\infty\}$. We now define $F_R(f)$ by cases.
\begin{enumerate}
\item If $\{m^f_k\}_{k < N^f}$ is an $<_R$-descending sequence, then there is a least $k > 0$ such that $m^f_k >_R m^f_{k-1}$. We let $F_R(f) = m^f_k + 1$.
\item If $\{m^f_k\}_{k < N^f}$ is an infinite $<_R$-descending sequence, then $F_R(f): = 0$
\item If $\{m^f_k\}_{k < N^f}$ is finite, nonempty, and $<_R$-descending, then $m^f_k$ is the largest number in this set and  $F_R(f) := m^f_k + 1$.
\item If $\{m^f_k\}_{k < N^f} = \emptyset$, i.e.\ $f$ is constant 1, then $F_R(f) := 0$.
\end{enumerate}
Now let $\{f_\gamma\}_{\gamma \leq \alpha}$ be the $<_L$-increasing sequence constructed through the evaluation of $\Theta_0(F_R)$. If for some $\gamma$ we define $F_R(f_\gamma)$ via items (2) or (4), we have that $\alpha = \gamma$ and that $\Theta_0(F_R) = \{f_\alpha\}$.

\smallskip

On the other hand, if we define $F_R(f_\gamma)$ via items (1) or (3), then consider the corresponding $m^{f_\gamma}_k$ where $F_R(f_\gamma) = m^{f_\gamma}_k + 1$. We then have that
\begin{itemize}
\item $f_{\gamma + 1}(n) = f_\gamma(n)$ for $n < m^{f_\gamma}_k$,
\item $f_{\gamma + 1}(n) = 1$ and $f_\gamma(n) = 0$ for $n = m^{f_\gamma}_k$,
\item $f_{\gamma + 1}(n) = 0$ for $n > m^{f_\gamma}_k$.
\end{itemize}
Thus the process cannot stop in any of those cases. Then the theorem follows from the following claim \eqref{DGG} for all $f\in C$, which we prove by induction on $\gamma \leq \alpha$.
\be\label{DGG}
\textup{If $\{m^f_k\}_{k < N^f}$ is an infinite $<_R$-descending sequence, then $f_\alpha \leq_L f$.}
\ee
For $\gamma = 0$, \eqref{DGG} is trivial, and for $\gamma$ a limit ordinal, the induction step is trivial.
\newline
So assume that the induction hypothesis holds for $\gamma$ and that $f$ is such that $\{m^f_k\}_{k < N^f}$ is an infinite $<_R$-descending sequence. Since $f_\gamma \leq_L f$, we cannot have that $f_\gamma$ is the constant 1. Further, if $f_\gamma$ codes an infinite descending $<_R$-sequence, then the process stops, and there is no $f_{\gamma + 1}$. So the interesting cases are the cases where either item (1) or item (3) holds in the definition of $F_{R}$.

\smallskip

If $\{m^{f_\gamma}_k\}_{k < N^{f_\gamma}}$ is not a descending sequence, let $k$ be as in the definition of $F_R(f_\gamma)$. If $f(n) > f_\gamma(n)$ for some least $n < m_k^{f_\gamma}$, we also have that $f(n) > f_{\gamma + 1}(n)$ for the same least $n$, and the induction hypothesis is preserved.

\smallskip

If $f(n)=  f_\gamma(n)$ for all $n < m_k^{f_\gamma}$, we must have that $f_\gamma(m_k^{f_\gamma}) = 0$, by the choice of $m_k^{f_\gamma}$ in this case, and that $f(m_k^{f_\gamma}) = 1$, since otherwise $f$ 
 would not even code a descending sequence, and then not an infinite one, as assumed.  Then it is clear that $f_{\gamma + 1} \leq_L f$ as well.

\smallskip

Now assume that $F_R(f_\gamma) = f_\gamma(m_k^{f_\gamma})$ due to item (3) in the definition of $F_R$. This requires that $f_\gamma(n) = 1$ for all $n > m_k^{f_\gamma}$. Then we use the universal formulation of the induction hypothesis to see that $m_k^{f_\gamma}$ must be in the well-ordered initial segment of $<_R$, since otherwise there would be an infinite descending sequence continuing the finite sequence coded by $f_\gamma$, and this sequence can again be coded by some $f'$ below $f_\gamma$ in $<_L$, contradicting the induction hypothesis.

\smallskip

But then, we must have that $f(m_k^{f_\gamma}) = 1$, since this function only can take the value 0 in the non-well-ordered part, and we can argue as in the previous case.
This ends the proof of the claim.

\smallskip

As a consequence of the claim, we see that $\Theta_0(F_R)$ will give us the leftmost infinite descending sequence, if there is one, and the constant 1 if there are none. Thus we can use $\Theta_0(F_R)$ to decide if $<_R$ is a well-ordering or not.\end{proof}
\begin{cor} The Suslin functional $\bf S$ is computable in Borel's $\Theta$. \end{cor}
\begin{proof}Given $f$, $\bf S$ decides if $(\forall g^{1})( \exists n^{0})( f(\bar gn) = 0)$ or not. This is computably equivalent to asking if the Kleene-Brouwer ordering of a certain tree is a well-ordering or not, a problem decidable by Borel's $\Theta$. \end{proof}

\subsection{Weak versus the special fan functionals}\label{kood}
We construct a particular functional $\Lambda_{0}$ satisfying $ \WCF(\Lambda_{0})$ and which produces hyperarithmetical output for hyperarithmetical input.
By Theorem \ref{refine}, the functional $\Lambda_{0}$ cannot be a $\Theta$-functional. In \cite{dagsamII} there is a stronger theorem, with a  more complex proof. 
We include the construction below partly because it is less of an ad hoc construction of a weak fan functional and partly because it illustrates how the  Sacks basis theorem is used.

\smallskip

We warn the reader that due to the rewrite of Section \ref{thetatje}, the numbering in this section has been changed from \cite{dagsam.I.p}.
The numbering in the latter was used when writing e.g.\ \cite{dagsamII}.
We first prove the following consequence of the Sacks Basis Theorem; we refer to \cite{Sacks.high}*{IV.2} for an account of the latter.
\begin{thm}\label{lem.sacks}
For every hyperarithmetical function $G^{2}$, the set $\bigcup_{f} C_{\overline{ f}G(f)}$ has measure 1, 
\emph{where $f$ ranges over the binary hyperarithmetical functions}.
\end{thm}
\begin{proof}
The \emph{Sacks Basis Theorem} is the following statement (\cite{Sacks.high}*{p.\ 93}):
\begin{quote}
\emph{If D is a hyperarithmetical set of functions of positive measure, then D contains a hyperarithmetical element.}
\end{quote}
Let $G^{2}$ be hyperarithmetical, let $\m$ be the standard measure on Cantor space, and let $\epsilon > 0$ be given. It suffices to prove that the set above has measure $> 1 - \epsilon$.
To this end, let $n$ be so large that $\m(\{f : G(f) < n\}) > 1-\epsilon$. Let $S_n$ be the set of sequences $s$ of length $n$ such that $C_s$ intersected with the set above has positive measure. By the basis theorem, each set $C_s$ will contain a hyperarithmetical $f$ with $G(f) < n$ whenever $s \in S_n$, and the union of these sets $C_s$ has measure $> 1 - \epsilon$.
\end{proof}

We now define, based on G\"odel's constructible universe $L$ relativized to any functional $G^{2}$, an explicit construction of a specific weak (and a special) fan functional. 
\bdefi[The functionals $\Lambda_{0}$ and $\Theta_{0}$]
We let $L_\alpha[G]$ be level $\alpha$ in the constructible universe relativized to $G^{2}$, where we have added a symbol for the functional $G$ to the language of set theory.
\begin{enumerate}
\item In order to ``compute" $\Lambda_{0}(G,k)$, first find the least ordinal $\alpha$ such that $\m\big(\bigcup_{f \in L_\alpha[G]}C_{\overline{f}G(f)} \big)>1- \frac{1}{k}$, 
and then use the $G$-definable well-ordering of $L_\alpha[G]$ to select a finite list of $f$'s doing the job.
\item In order to ``compute" $\Theta_{0}(G)$, continue  the process above until we have a covering of Cantor space.
\end{enumerate}
We have to prove that this process will go on until we have a covering of Cantor space, by proving that unless we have a covering at stage $\alpha$, there is an element of $L_{\alpha + 1}[G]$ not covered by the open set $O_\alpha$ considered at stage $\alpha$. This is trivial, since $L_\alpha[G] \in L_{\alpha+1}[G]$, and then the leftmost function not covered by $O_\alpha$ is definable, and thus an element of $L_{\alpha + 1}[G]$.
\edefi
The definition of $\Theta_{0}$ constitutes (in a technical sense) the optimal way of computing a special fan functional, as will be explored in future research.  
\begin{cor}\label{essje} 
Let $\Lambda_{0}$ be as constructed above, and let $G$ be a total, hyperarithmetical function of type 2. Then $\Lambda_{0}(G,k)$ is a finite list of hyperarithmetical functions.
Indeed, there is a partial $\Lambda^- \subseteq \Lambda_0$ that is computable in $\exists^2$ and that terminates on all total $G$ computable in $\exists^2$.
\end{cor}
\begin{proof} 
If $G^{2}$ is hyperarithmetical and $\alpha$ is a computable ordinal, $L_\alpha \subseteq L_\alpha[G] \subseteq L_{\omega_1^{CK}}.$ 
By Theorem \ref{lem.sacks} the search for a value of $\Lambda_{0}(G,k)$ will end at a computable ordinal, and the output is hyperarithmetical. 
By Gandy selection, the process evaluating $\Lambda_0(G)$ is computable in $\exists^2$ to the extent it terminates below $\omega_1^{\textup{CK}}$.
\end{proof}

\section{Explosions and non-explosions}\label{weako}
An `explosion' refers to two logical principles (or functionals) that are relatively weak in isolation, but much stronger when combined.  
We show that $\STP$ gives rise to an explosion when combined with $\paai$, while $\LMP$ is shown to yield no such explosion.   
These results are partly based on explosions (resp.\ non-explosions) involving $\Theta$-functionals (resp.\ $\Lambda$-functionals) from the previous section.
We also study the relation of $\Theta$-functionals to other explosive functionals. 

\subsection{Introduction}
We proved in Section \ref{thetatje} that Borel's $\Theta$ computes the Suslin functional in combination with $\exists^{2}$.  
By contrast, we proved in Section \ref{kood} that there is a $\Lambda$-functional that produces hyperarithmetical output for hyperarithmetical input (i.e.\ computable in $\exists^{2}$).  
Thus, $\Theta$-functionals seem to be relative strong, while $\Lambda$-functionals seem to be (or can be) rather weak. 
Based on the connection between Nonstandard Analysis and computability theory, the aforementioned results suggest that $\paai+\STP$ and $\paai+\LMP$ are resp.\ quite strong and relatively weak, all compared to say $\exists^{2}$.  

\smallskip

This hunch turns out to be correct: we show in this section that $\paai+\STP$ implies $\ATR$ relative to `\st' while $\paai+\LMP$ does not.  In other words, $\STP$ is explosive when combined with $\paai$, while $\LMP$ is not; note however that $\STP$ and $\LMP$ (and $\WKL$ and $\WWKL$) are `quite close', as discussed in Remark~\ref{triplex}.  
 Furthermore, Corollary~\ref{frigjr} provides a (positive) answer to Hirschfeldt's question (see \cite{montahue}*{\S6.1}) concerning equivalences in RM which require a stronger base theory.  
 
 \smallskip
 
Finally, we discuss the connection between special fan functionals and Kohlenbach's generalisations of weak K\"onig's lemma in Section \ref{kokkon}.  These results show that special fan functionals can be viewed as a version of the classical fan functional \emph{with nonstandard continuity} instead of the epsilon-delta variety.  

\subsection{Transfinite recursion and nonstandard compactness I}
We prove the main \emph{negative} result of this section, namely that $\P_{0}+\paai+\LMP$ does not prove $\ATR_{0}^{\st}$.
Regarding definitions, $\ATR_{0}$ is $\ACA_{0}$ plus the second-order schema:
\be\tag{$\ATR_{\theta}$}
(\forall X^{1})\big[\WO(X)\di (\exists Y^{1})H_{\theta}(X, Y) \big], 
\ee
for any arithmetical $\theta$, and where $\WO(X)$ expresses that $X$ is a countable well-ordering and $H_{\theta}(X, Y)$ expresses that $Y$ is the result from iterating $\theta$ along $X$.  
More details and related results may be found in \cite{simpson2}*{V.2}.

\smallskip
Secondly, to gain some intuitions regarding $\paai$ and $\ATR_{0}$, we  list a few facts which are merely the nonstandard analogues of well-known results, and thus readily proved.  
For instance, an early theorem of higher-order computability theory going back to Kleene (see \cite{longmann}*{Theorem 5.4.1} or \cites{Sacks.high, Rogers}) states that the functions computable in $\exists^{2}$ are exactly the $\Delta_{1}^{1}$ (or \emph{hyperarithmetical}) functions.  The nonstandard counterpart of $\exists^{2}$ (actually the equivalent $\mu^{2}$) is $\paai$ and we thus expect that $\P+\paai$ can prove comprehension for $\Delta_{1}^{1}$-sets (relative to `st').  
This suspicion turns out to be correct, as follows. 
\begin{thm}\label{umdidi} 
The system $\P_{0}+\paai$ proves $\big(\Delta_{1}^{1}\textup{\textsf{-CA}}\big)^{\st}$, i.e.\ we have for all standard $f^{1}, g^{1}$ that
\begin{align}
(\forall^{\st}n^{0})\big[(\exists^{\st}k^{1})&(\forall^{\st}m^{0})(f(\overline{k}m,n)=0) \asa(\forall^{\st}l^{1})(\exists ^{\st}r^{0})(g(\overline{l}r,n)\ne0) \big]\label{noorjega2}\\
& \di (\exists^{\st}h^{1})(\forall^{\st}n)\big[ (\exists^{\st}k^{1})(\forall^{\st}m^{0})(f(\overline{k}m,n)=0) \asa h(n)=0   \big].\notag
\end{align}
\end{thm}
\begin{proof}
We only provide a sketch of the proof.  
First of all, we can obtain $(\mu^{2})^{\st}$ from $\paai$ by applying $\HAC_{\INT}$ to \eqref{trakke}.  
Now use this standard version of Feferman's mu to remove the type zero quantifiers (with variables $m^{0}, r^{0}$) in the equivalence from the antecedent of \eqref{noorjega2}.  
Consider the reverse implication of the resulting formula and apply $\HAC_{\INT}$.  The resulting functional, combined with $(\mu^{2})^{\st}$, now readily yields the function $h$ from the consequent of \eqref{noorjega2}.  
\end{proof}
Thirdly, $\paai$ does not really provide anything beyond the hyperarithmetical, which is suggested by the following result.  
\begin{thm}\label{weasy}
Assuming it is consistent, $\P+\paai$ does not prove $\ATR^{\st}_{0}$.
\end{thm}
\begin{proof}
Suppose $\P+\paai$ does prove $\ATR^{\st}_{0}$.  We shall focus on the latter for the special case $\theta_{0}(n, Y)$ expressing that $n^{0}$ is an element of the Turing jump of $Y^{1}$.     
Hence, $\P+\paai$ proves
\be\label{prilleke}
(\forall^{\st}X^{1})\big[[\WO(X)]^{\st}\di (\exists^{\st} Y^{1})[H_{\theta_{0}}(X, Y)]^{\st}  \big].
\ee
As noted in \cite{simpson2}*{V.2.2}, $H_{\theta}$ is arithmetical if $\theta$ is.  Hence, $[H_{\theta_{0}}(X, Y)]^{\st} \asa H_{\theta_{0}}(X, Y)$ for standard $X, Y$ thanks to $\paai$.  
Similarly, $\WO(X)\di [\WO(X)]^{\st}$ for standard $X$ using $\paai$, and \eqref{prilleke} thus implies
\be\label{prilleke2}
(\forall^{\st}X^{1})\big[\WO(X)\di (\exists^{\st} Y^{1})H_{\theta_{0}}(X, Y)  \big], 
\ee
where the only `st' inside the square brackets is with the $Y$-quantifier.  
Clearly, \eqref{prilleke2} has a normal form and applying Theorem \ref{consresultcor} to $\P\vdash [\paai\di \eqref{prilleke2}]$, we obtain a term $t$ such that $\textsf{E-PA}^{\omega*}$ proves 
\be\label{dirkske2}
(\forall \mu^{2})\Big[\MU(\mu)\di (\forall X^{1})\big[\WO(X)\di (\exists  Y^{1}\in t(X, \mu))H_{\theta_{0}}(X, Y)  \big]\Big].  
\ee
We now derive a contradiction from \eqref{dirkske2}:  By the latter, $\EPA^{\omega*}+(\mu^{2})$ proves
\be\label{ieteninks}
(\forall X^{1})\big[\WO(X)\di (\exists  Y^{1})H_{\theta_{0}}(X, Y)  \big],
\ee
which is equivalent to a $\Pi_{2}^{1}$-formula since $\WO(X)$ is $\Pi_{1}^{1}$ and the consequent of \eqref{ieteninks} is $\Sigma_{1}^{1}$.   
However, the conservation result in \cite{yamayamaharehare}*{Theorem 2.2} implies that $\ACA_{0}$ and $\textsf{E-PA}^{\omega}+ \QFAC^{1,0}+{(\mu^{2})}$ prove the same $\Pi^{1}_{2}$-formulas.
But \eqref{ieteninks} implies the existence of the $\omega$-th Turing jump, which is not provable in $\ACA_{0}$ by \cite{simpson2}*{I.11.2}, a contradiction.  Alternatively, since $\textsf{HYP}$, the model consisting of all hyperarithmetical sets (see e.g.\ \cite{simpson2}*{V} for details on this model), is a model of $\ACA_{0}$, \eqref{ieteninks} holds in $\textsf{HYP}$, which is impossible as shown in the proof of \cite{simpson2}*{V.2.6}.    
\end{proof}
Clearly, the previous proof also goes through for any $\Pi_{2}^{1}$-formula not provable in $\ACA_{0}$ (instead of $\ATR_{0}$).  
Next, we prove one of the main theorems of this section.  
\begin{thm}\label{rejeji}
Given its consistency, $\P+\paai+\LMP$ cannot prove $\ATR_{0}^{\st}$.
\end{thm}
\begin{proof}
First of all, we sketch an interesting aspect of well-orderings relating to the model $\textsf{HYP}$.  
As shown in \cite{simpson2}*{VIII}, $\textsf{HYP}$ is not a model of $\ATR_{0}$.  In particular, $\theta_{0}$ from the proof of Theorem \ref{weasy} satisfies (see \cite{simpson2}*{V.2.6}): 
\be\label{hypje}
\textsf{\textup{HYP}}\models (\exists X_{0}^{1})\big[\WO(X)\wedge (\forall  Y^{1})\neg H_{\theta_{0}}(X, Y)  \big].
\ee
It is important to note that $X_{0}^{1}$ from \eqref{hypje} is not necessarily a well-ordering: As studied in \cite{harry}, there exist (Turing computable) \emph{pseudo-well-orderings} which have no \emph{hyperarithmetical} infinite descending sequences but which nonetheless do have \emph{non-hyperarithmetical} infinite descending sequences.  In colloquial terms, the model $\textsf{HYP}$ ` thinks' that a pseudo-well-ordering is a well-ordering, while it is not.  

\smallskip
Secondly, to accommodate the previous observation regarding these pseudo-well-oderings, a slight tweak is needed to the proof of Theorem \ref{weasy}, as follows: Let $\WO(g, X)$ be the (arithmetical) formula expressing that $g^{1}$ is not an infinite descending sequence through $X$, i.e.\ $(\forall g^{1})\WO(g, X)$ is just the familiar $\WO(X)$.  Using $\paai$, we observe that $[\WO(X)]^{\st}$ follows from $(\forall^{\st}g^{1})\WO(g,X^{1})$ for standard $X$ (and is actually equivalent). 
Now suppose $\P+\paai+\LMP$ does prove $\ATR^{\st}$ and obtain, like in the previous proof, that
\be\label{prilleke22}
(\forall^{\st}X^{1})\big[(\forall^{\st}g^{1})\WO(g,X)\di (\exists^{\st} Y^{1})H_{\theta_{0}}(X, Y)  \big].
\ee
Now bring outside all standard quantifiers in \eqref{prilleke22} to obtain the following:
\be\label{prilleke23}
(\forall^{\st}X^{1})(\exists^{\st}g^{1}, Y^{1})\big[\WO(g,X)\di H_{\theta_{0}}(X, Y)  \big].
\ee
Applying Theorem \ref{consresultcor} to `$\P+\LMP+\paai\vdash \eqref{prilleke23} $', we obtain terms $i, o$ such that \textsf{E-PA}$^{\omega*}$ 
(and hence also any extension, like $\ZFC$) proves that:
\begin{align}\label{dirkske3}
(\forall \mu^{2}, \Lambda^{3}, X^{1})&\big[[\MU(\mu)\wedge \WCF(\Lambda)] \\
&\di \big[(\forall g\in i(X, \mu, \Lambda))\WO(g,X)\di (\exists  Y^{1}\in o(X, \mu, \Lambda))H_{\theta_{0}}(X, Y)  \big].  \notag
\end{align}
Now, by Theorem \ref{import24}, there exists (provable in $\ZFC$) an instance $\Lambda_{1}$ of the weak fan functional which from a functional computable in $\exists^{2}$ produces hyperarithmetical functions in a uniform way (computable in $\exists^2$).  
Furthermore, the functions computable in $\exists^{2}$ (and thus Feferman's mu) are the \emph{hyperarithmetical} ones. 

\smallskip
Finally, fix some Turing computable pseudo-well-ordering $X_{1}$ (as introduced in the first part of this proof).  By the choice of inputs, $i(X_{1}, \mu, \Lambda_{1})$ and $o(X_{1}, \mu, \Lambda_{1})$ from \eqref{dirkske3} are both finite sequences of hyperarithmetical functions.  Hence, the correct $Y^{1}\in o(X_{1}, \mu, \Lambda_{1})$ from \eqref{dirkske3} is hyperarithmetical, while the antecedent $(\forall g\in i(X_{1}, \mu, \Lambda_{1}))\WO(g,X_{1})$ of \eqref{dirkske3} holds by the assumption that $X_{1}$ has no infinite descending sequences which are also hyperarithmetical.  
However, by \cite{simpson2}*{V.2.6 and VIII.3.23}, there is no hyperarithmetical $Y$ such that $H_{\theta_{0}}(X_1, Y)$.  
Hence, \eqref{dirkske3} yields a contradiction, 
 thanks to the existence of Turing computable pseudo-well-orderings and the weak fan functional $\Lambda_{1}$ from Theorem~\ref{import24}.
\end{proof}
Clearly, the previous proof also goes through for other sentences (than $\ATR$) false in the model $\HYP$.  As a result, the system from the theorem is consistent if $\ACA_{0}$ is, a rather mild assumption in the grand scheme of things.  
While $\WKL_{0}$ and $\WWKL_{0}$ are `rather close' in the sense of logical strength, we next prove that $\paai+\STP$ behaves 
very differently in that it does imply $\ATR_{0}^{\st}$.  

\subsection{Transfinite recursion and nonstandard compactness II}\label{clopen}
We prove the main positive result of this section, namely we obtain $\ATR_{0}^{\st}$ from $\paai+\STP$.  
This result should be contrasted with $\paai+\LMP$ and Theorem~\ref{rejeji} from the previous section, 
\begin{thm}\label{mikeh}
The system $\P_{0}+\paai+\STP$ proves $\ATR_{0}^{\st}$.
\end{thm}
\begin{proof}
As shown in \cite{simpson2}*{V.5.1}, $\RCA_{0}$ proves that $\ATR_{0}$ is equivalent to $\Sigma_{1}^{1}\mSEP$; the latter is defined as: For $\varphi_{1}, \varphi_{2}\in \Sigma_{1}^{1}$ not involving the variable $Z^{1}$, we have
\be\label{heffer}
(\forall n^{0})(\neg\varphi_{1}(n)\vee \neg\varphi_{2}(n))\di (\exists Z^{1})(\forall n^{0})\big(\varphi_{1}(n)\di n\in Z\wedge \varphi_{2}(n)\di n\not\in Z \big).
\ee
We shall prove $[\Sigma_{1}^{1}\mSEP]^{\st}$ in $\P_{0}+\paai+\STP$.  Since $\P_{0}$ proves the axioms of $\RCA_{0}$ relative to `st', we therefore obtain $\ATR_{0}^{\st}$.  
Now let $\varphi_{i}(n)$ be short for the formula $(\exists g^{1}_{i})(\forall x_{i}^{0})(f_{i}(\overline{g_{i}}x_{i}, n)=0)$ and fix standard $f_{i}^{1}$ for $i=1,2$.  Then assume $\big[(\forall n^{0})(\neg\varphi_{1}(n)\vee \neg\varphi_{2}(n))\big]^{\st}$, which is the formula
\[
(\forall^{\st}n^{0})
\big[ 
(\forall^{\st} g^{1}_{1})(\exists^{\st} x_{1}^{0})(f_{1}(\overline{g_{1}}x_{1}, n)\ne 0) 
\vee  
(\forall^{\st} g^{1}_{2})(\exists^{\st} x_{2}^{0})(f_{2}(\overline{g_{2}}x_{2}, n)\ne 0)
\big].
\]
For fixed nonstandard $N^{0}$, the previous formula implies (without using $\paai$):
\be\label{lahiel}
(\forall^{\st}n^{0}, g_{1}^{1}, g_{2}^{1})
\big[ 
(\exists  x_{1}^{0}\leq N)(f_{1}(\overline{g_{1}}x_{1}, n)\ne 0) 
\vee  
(\exists  x_{2}^{0}\leq N)(f_{2}(\overline{g_{2}}x_{2}, n)\ne 0)
\big].
\ee
Let $A_{i}(n, g_{i})$ be the (equivalent to quantifier-free) formula $(\exists  x_{i}^{0}\leq N)(f_{i}(\overline{g_{i}}x_{i}, n)\ne 0) $ and let $A(n, g_{1}, g_{2})$ be the formula $A_{1}(n, g_{1})\vee A_{2}(n, g_{2})$, i.e.\ the formula in square brackets in \eqref{lahiel}.  By assumption, $(\forall^{\st}n^{0}, g_{1}^{1}, g_{2}^{1})A(n,g_{1},g_{2})$.  
Now consider:
\begin{align}\label{ideaal}
(\forall^{\st} v^{1^{*}}, x^{0^{*}})(\exists & w^{1^{*}}, y^{0^{*}})(\forall g^{1} \in v, n^{0}\in x)\\
&\big[ g\in w \wedge n\in y \wedge (\forall h_{1}, h_{2}\in w, m\in y)A(m,h_{1}, h_{2}) \big],\notag
\end{align}
which holds by taking $w=v$, $y=x$.  Applying \emph{Idealisation} \textsf{I} to \eqref{ideaal}, we obtain  
\be\label{forgik}
(\exists w^{1^{*}}, y^{0^{*}})(\forall^{\st} g^{1}, n^{0})\big[ g\in w \wedge n\in y \wedge (\forall h_{1}, h_{2}\in w, m\in y)A(m,h_{1}, h_{2}) \big], 
\ee
which -intuitively speaking- provides two sequences $w, y$ (of nonstandard length) encompassing all standard functions and standard numbers and such that all of its elements satisfy $A$.  
In particular, one can view \eqref{forgik} as obtained by applying overspill to \eqref{lahiel} while making sure all standard functions are in $w$.  

\smallskip
Next, define the set $Z_{0}^{1}$ (actually a binary sequence) as follows:  $n\in Z_{0}\asa (\exists g_{1}\in w)\neg A_{1}(n,g)$, where $w$ is the sequence from \eqref{forgik}.  
Note that the right-hand side of the equivalence is actually `$(\exists i^{0}<|w|)\neg A_{1}(n, w(i))$', i.e.\ $Z_{0}$ is definable in $\P_{0}$.    

\smallskip
Let $Z^{1}$ be a standard set such that $Z_{0}\approx_{1} Z$  as provided by $\STP$.  Furthermore, $\paai$ establishes the following implications (for standard $n$):
\begin{align*}
(\exists^{\st} g^{1}_{1})(\forall^{\st} x_{1}^{0})(f_{1}(\overline{g_{1}}x_{1}, n)=0) 
&\di (\exists^{\st} g^{1}_{1})(\forall  x_{1}^{0}\leq N)(f_{1}(\overline{g_{1}}x_{1}, n)=0)\\
&\di (\exists g^{1}_{1}\in w)(\forall  x_{1}^{0}\leq N)(f_{1}(\overline{g_{1}}x_{1}, n)=0)\\
&\di (\exists g^{1}_{1}\in w)\neg A_{1}(n,g_{1})\di n\in Z_{0}\di n\in Z.
\end{align*}
Note that $\paai$ is (only) necessary to establish the first implication.  Now, since $y$ from \eqref{forgik} contains all standard numbers, the second conjunct of \eqref{forgik} implies (by definition)
that for standard $m$ (by the definition of $A$): 
\be\label{fronxi}
(\forall h_{1}\in w)A_{1}(m,h_{1})\vee (\forall h_{2}\in w)A_{2}(m,h_{2}).
\ee
Similarly, consider the following series of implications (for standard $n$):
\begin{align}
(\exists^{\st} g^{1}_{2})(\forall^{\st} x_{2}^{0})(f_{2}(\overline{g_{2}}x_{2}, n)=0) 
&\di (\exists^{\st} g^{1}_{2})(\forall  x_{2}^{0}\leq N)(f_{2}(\overline{g_{2}}x_{2}, n)=0)\notag\\
&\di (\exists g^{1}_{2}\in w)(\forall  x_{2}^{0}\leq N)(f_{2}(\overline{g_{2}}x_{2}, n)=0)\notag\\
&\di (\exists g^{1}_{2}\in w)\neg A_{2}(n,g_{2})\label{hoi2}\\
&\di (\forall g^{1}_{1}\in w) A_{1}(n,g_{1})\label{hoi1}
\di n\not\in Z_{0}\di n\not\in Z.
\end{align}
Note that $\paai$ is (only) necessary to establish the first implication, while \eqref{hoi1} follows from \eqref{hoi2} by \eqref{fronxi}.  
Thus, we observe that $Z$ is as required for $\Sigma_{1}^{1}$-comprehension \eqref{heffer} relative to `st', and we are done.  
\end{proof}
Note that the previous proof makes essential use of $\STP$ to obtain $Z$ from $Z_{0}$ as $w$ from \eqref{forgik} is nonstandard, i.e.\ $\WKL^{\st}$ does not suffice.  
Furthermore, the previous proof seems to go through in the constructive system $\H$ from \cite{brie}, as well as in $\P_{0}$ without the axiom of extensionality \eqref{EXT}.   
We also note that the particular use of \emph{Idealisation} to obtain \eqref{forgik} from \eqref{ideaal} is inspired by \cite{benno2}.
We now discuss some more interesting aspects of the previous proof.
\begin{rem}[The power of Nonstandard Analysis]\rm
Comparing the previous proof to that of $\Sigma_{1}^{1}\mSEP$ in \cite{simpson2}*{V.5}, the proof in Nonstandard Analysis is much shorter and \emph{conceptually much simpler}.  
This may be explained as follows:
It is often said that `one can search through the naturals, but not through the reals (or Baire space)'.   The previous proof showcases a powerful feature of Nonstandard Analysis:  Thanks to the sequence $w$ from \eqref{forgik}, we \emph{can} search through the \emph{standard} reals (standard functions of Baire space) in a specific sense.  
Thanks to this `search' feature of Nonstandard Analysis, the previous proof is very similar\footnote{To prove that $\STP$ implies $[\Sigma_{1}^{0}\mSEP]^{\st}$, apply overflow (which is an instance of \emph{Idealisation}) to $[(\forall n^{0})(\neg\varphi_{1}(n)\vee \neg\varphi_{2}(n))]^{\st}$ for $\varphi_{i}(n)\equiv (\exists n^{0}_{i})(f_{i}(n, n_{i})=0)$, 
and define the set $Z_{0}$ by $n\in Z_{0}\asa (\exists n^{0}_{1}\leq N_{0})f_{1}(n, n_{1})=0$ where $N_{0}$ is the number obtained by overflow.  Applying $\STP$ to $Z_{0}$ finishes the proof.\label{klukkerke}} to that $\STP$ implies $[\Sigma_{1}^{0}\mSEP]^{\st}$ as in Footnote \ref{klukkerke}.  Hence, the similarities between $\WKL$ and $\ATR_{0}$, from \cite{simpson2}*{I.11.7}, also exist in Nonstandard Analysis.  Finally, we point out that by \cite{simpson2}*{V.5.1}, a single application of $\Sigma_{1}^{1}\mSEP$ provides the set $Y$ from $\ATR_{0}$.        
\end{rem}
We now discuss a number of interesting corollaries.  
\begin{cor}\label{fokkers}
There are terms $i, o$ of G\"odel's $T$ such that $\EPRA^{\omega*}$ proves 
\begin{align}\label{fogi}
(\forall \mu^{2}, \Theta^{3})&\big[[\MU(\mu)\wedge \SCF(\Theta)] \\
&\di (\forall X^{1})\big[(\forall g\in i(X, \mu, \Theta))\WO(g,X)\di (\exists  Y^{1}\in o(X, \mu, \Theta))H_{\theta_{0}}(X, Y)  \big].  \notag
\end{align}
where $\theta_{0}(n, Z)$ expresses that $n^{0}$ is a member of the Turing jump of $Z^{1}$.   
\end{cor}
\begin{proof}
Immediate following the proof of Theorem \ref{rejeji}.
\end{proof}
The following corollary has the advantage that it `directly' establishes that $\paai\not\di \STP$, but the disadvantage is that it does not generalise to $\Paai$.  
\begin{cor}\label{dagdag}
The system $\EPRA^{\omega*}+(\mu^{2})+(\exists \Theta)\SCF(\Theta)$ proves $\ATR_{0}$.\\
Assuming the system is consistent, $\P+\paai+\LMP$ cannot prove $\STP$.  
\end{cor}
\begin{proof}
The first part is immediate from \eqref{fogi}.  For the second, part, if $\P+\paai+\LMP$ could prove $\STP$, then it would also prove $\ATR^{\st}$ by the theorem, but this impossible by Theorem \ref{rejeji}.    
\end{proof}
The following corollary proves results analogous to Theorem \ref{refine}; the latter is proved using computability theory while the former follows from Nonstandard Analysis.  
Both approaches have pros and cons: Theorem \ref{refine} requires a tricky construction which however does give rise to additional information, namely a $\Theta$-functional in which the hyper-jump is computable.  
The approach using Nonstandard Analysis avoids the tricky construction needed in the computability theoretic approach, but does not tell us anything about the hyper-jump.
\begin{cor}\label{cor.alt.6.8}
Let $\Theta$ be such that $\SCF(\Theta)$.
There is $G^{2}$ computable in $\exists^{2}$ such that $\Theta(G)$ is not hyperarithmetical.  
\end{cor}
\begin{proof}
Suppose $\Theta_{1}$ satisfying $\SCF(\Theta_{1})$ is such that $\Theta_{1}(g)$ is hyperarithmetical for all $g^{2}$ computable in $\exists^{2}$.  
Without loss of generality we may assume that $\Theta_1$, restricted to the hyperarithmetical functions of type 2, is partially computable in $\exists^2$, by the following argument:
by assumption, for every hyperarithmetical $g^{2}$ there is hyperarithmetical $\langle f_1 , \ldots , f_k\rangle$ that yields an open covering of Cantor space via $g$. By Gandy selection, we may search for one such sequence uniformly computable in $\exists^2$. We may then construct $\Theta_2$ agreeing with $\Theta_1$ on non-hyperarithmetical inputs, and with the result of this search on hyperarithmetical input. We have $\SCF(\Theta_2)$ and $\Theta_2$ satisfies our extra assumption.

\smallskip
Applying \eqref{fogi} for a pseudo-well-ordering $X_{1}$ (as discussed in the proof of Theorem \ref{rejeji}), we obtain a contradiction as in the proof of Theorem \ref{rejeji}.  
Indeed, in this case, $i(X_{1}, \mu, \Theta_{1})$ and $o(X_{1}, \mu, \Theta_{1})$ are finite sequences of hyperarithmetical functions, and hence $(\forall g\in i(X, \mu, \Theta))\WO(g,X_{1})$ 
holds as $X_{1}$ is a pseudo-well-ordering.  But there is no hyperarithmetical $Y$ such that $H_{\theta_{0}}(X, Y) $, as discussed in the proof of Theorem \ref{rejeji}, i.e.\ \eqref{fogi} implies a contradiction.
\end{proof}
The previous corollary also follows from Theorem \ref{refine} and its corollary, but the previous proof is interesting in its own right. 

\smallskip

The following corollary strengthens the above results slightly.  Let $\con(S)$ be the usual $\Pi_{1}^{0}$-sentence expressing the consistency of the system $S$ (see e.g.\ \cite{simpson2}*{II.8.2}).
\begin{cor}\label{konk}
The systems $\P+\paai+\STP$ and  $\EPA^{\omega*}+(\mu^{2})+(\exists \Theta^{3})\SCF(\Theta)$ prove the consistency of $\ATR_{0}$, i.e.\ $\con(\ATR_{0})$.
\end{cor}
\begin{proof}
By definition, $\P$ includes external induction \textsf{IA}$^{\st}$, and hence $[\Sigma_{1}^{1}\textsf{-IND}]^{\st}$. 
By the theorem, $\P+\paai+\STP$ proves $[\ATR_{0}+\Sigma_{1}^{1}\textsf{-IND}]^{\st}$.  However, \cite{simpson2}*{IX.4.7} states that $\ATR_{0}+\Sigma_{1}^{1}\textsf{-IND}$ proves $\con(\ATR_{0})$.  Since consistency statements are $\Pi_{1}^{0}$ and since $\P$ proves the axioms of $\RCA_{0}$ relative to `st', we observe that $\P+\paai+\STP\vdash \con(\ATR_{0})$.  Applying term extraction yields the corollary.  
\end{proof}
Next, we discuss the `explosion' in our above results.  
\begin{rem}[Explosion of strength]\rm\label{triplex}
As shown above, the difference in strength between $\paai+\STP$ and $\paai+\LMP$ is significant, and the same holds for $\exists^{2}$ when combined with resp.\ $\Theta$-functionals and $\Lambda_{1}$.  
Now, $\STP$ and $\Theta$-functionals are based on $\WKL$, while $\LMP$ and $\Lambda_{1}$ are based on $\WWKL$.   However, to the best of our knowledge, there is no natural principle between $\WKL$ and $\WWKL$:  there is no principle between the latter two in the RM zoo (\cite{damirzoo}), and even in the highly fine-grained structure of the Weihrauch degrees, there is currently no known natural problem between $\WWKL$ and $\WKL$, as communicated to us by Vasco Brattka.  Thus, one can say that $\WWKL$ and $\WKL$ are `very close', but we nonetheless have a dramatic shift in strength for the associated $\paai+\STP$ and $\paai+\LMP$, and the same holds for $\exists^{2}$ when combined with resp.\ $\Theta$-functionals and $\Lambda_{1}$.  
\end{rem}
Finally, we obtain an interesting result in RM as follows: a small number of equivalences in RM are known to require a base theory \emph{stronger} than $\RCA_{0}$, and Hirschfeldt has asked whether there are more such equivalences (see \cite{montahue}*{\S6.1}). 

\smallskip 

We provide such an example based on our above results.
To this end, let $\Sigma_{1}^{1}\mSEP_{\ns}$ be \eqref{heffer}$^{\st}$ for $\varphi_{i}(n)\equiv(\exists g^{1}_{i})(\forall x_{i}^{0})(f_{i}(\overline{g_{i}}x_{i}, n)=0)$ and \emph{any} $f_{i}\leq_{1}1$. 
Thus, $\Sigma_{1}^{1}\mSEP_{\ns}$ is essentially just $[\Sigma_{1}^{1}\mSEP]^{\st}$ with the leading `st' in `$(\forall^{\st}f_{1}, f_{2}\leq_{1}1)$' removed.  Recall that $\STP$ is just $\WKL^{\st}$ with the leading `st' in `$(\forall^{\st}T\leq_{1}1)$' removed as in \eqref{fanns}.  The following is a corollary to Theorem \ref{mikeh}.
\begin{cor}\label{weird}
The system $\P_{0}+\paai$ proves $\STP\asa\Sigma_{1}^{1}\mSEP_{\ns}$, while $\P_{0}$ cannot prove $\STP\di \Sigma_{1}^{1}\mSEP_{\ns}$.  
\end{cor}
\begin{proof}
Regarding the first part, the forward implication follows from Theorem \ref{mikeh} if $[\Sigma_{1}^{1}\mSEP]^{\st}\di \Sigma_{1}^{1}\mSEP_{\ns}$.  
The latter implication follows by taking $f_{1}, f_{2}\leq_{1}1$ as in $\Sigma_{1}^{1}\mSEP_{\ns}$ and applying  $\STP$ to obtain \emph{standard} $f_{1}', f_{2}'$ such that $f_{1}'\approx_{1} f_{1}$ and $ f_{2}'\approx f_{2}$.  
Since $\Sigma_{1}^{1}\mSEP$ is a statement of second-order arithmetic, $f_{1}, f_{2}$ only occur as `$f_{1}(n)$' and `$f_{2}(n)$', and we may thus replace $f_{1}',f_{2}'$ by $f_{1},f_{2}$ in $[\Sigma_{1}^{1}\mSEP]^{\st}$, yielding the desired implication.  The reverse implication follows from applying $\Sigma_{1}^{1}\mSEP_{\ns}$  for $\varphi_{1}(n)\equiv (f(n)=0) $ and $\varphi_{2}(n)\equiv (f(n)=1)$ for given  $f\leq_{1}1$: The resulting standard $Z^{1}$ is such that $(\forall^{\st}n^{0})(f(n)=0\asa n\in Z)$, and the characteristic function of $Z$ yields the desired standard $g\leq_{1}1$ such that $f\approx_{1}g$.    
The second part follows from the fact that $\P_{0}+\STP$ is conservative over $\WKL_{0}$ and $\ATR_{0}$ is not.
\end{proof}
Corollary \ref{weird} could be dismissed as a curiosity, but Corollary \ref{frigjr} constitutes a challenge to the `Big Five' picture.  
We need a `trivially uniform' version of $\ATR_{0}$:
\be\tag{$\UATR$}
(\exists \Phi^{1\di 1})(\forall X^{1}, f^{1})\big[\WO(X)\di H_{f}(X, \Phi(X,f)) \big], 
\ee
where $H_{f}(X, Y)$ is just $H_{\theta}(X, Y)$ with $\theta(n, Z)$ defined as $(\exists m^{0})(f(n,m, \overline{Z}m)=0)$.
\begin{cor}\label{frigjr}
$\RCAo+(\exists \Theta)\SCF(\Theta)$ proves $(\mu^{2})\asa \UATR$; $\RCAo+\WKL$ doesn't.  
\end{cor}
\begin{proof}
The reverse implication is immediate.  
The non-implication is immediate as $\RCAo+(\mu^{2})$ is $\Pi_{2}^{1}$-conservative over $\ACA_{0}$ (\cite{yamayamaharehare}*{Theorem 2.2}) while $\RCAo+\UATR$ is not.  
The forward implication follows from Corollary \ref{fokkers}.  
Note that since $H_{\theta}(X, Y)$ is arithmetical if $\theta$ is, $\mu^{2}$ can select the correct $Y$ in \eqref{fogi}.  
\end{proof}
As noted in Section \ref{pampson}, $\STP$ seems to be \emph{robust}, i.e.\ equivalent to small perturbations of itself.  
The same of course holds for variations of $\Theta$, which suggests that the equivalence in the previous corollary is not a trick, but a robust result.  

\subsection{Generalisations of weak K\"onig's lemma}\label{kokkon}
We study the connection between $\Theta$-functionals and the functional $\kappa^{3}$ defined below, where the latter is based on Kohlenbach's axioms $\Phi_{n}$-$\WKL$ and $\Pi_{n}^{1,b}$-$\textsf{CA}$ from \cite{kohlenbach4}*{\S5-6}.  
Our motivation for this study is that both $\Theta$ and $\kappa$ give rise to conservative extensions of $\WKL_{0}$ \emph{in isolation} but become strong \emph{when combined} with $\mu^{2}$.   
We show that $\kappa$ computes $\Theta$-functionals but not vice versa, and that $\Theta$-functionals can be viewed as a version of the \emph{classical} fan functional with the role of `epsilon-delta' continuity replaced by nonstandard continuity.

\smallskip
First of all, we introduce $\kappa^{3}$, a higher-order version of Kohlenbach's $\Pi_{n}^{1,b}$-$\textsf{CA}$ from \cite{kohlenbach4}*{\S5-6}.  We first sketch the results regarding $\kappa^{3}$ while proofs are provided below.    
\be\label{KOT}\tag{$\kappa^3$}
(\exists \kappa^{2\di 1})(\forall Y^{2})\big[ (\exists f^{1}\leq_{1}1)(Y(f)=0) \di  Y(\kappa(Y))=0\big].
\ee
Two basic facts regarding $\kappa$ are that over the full type structure, this functional defines a choice operator for non-empty subsets of Cantor space, and we therefore cannot prove the existence of any instance of $\kappa$ in \textsf{ZF}.
\begin{rem}[Continuity, $\kappa$, and $\exists^{3}$]\label{misfit}\rm
Note that $\exists^3$ can decide any formula involving type zero and one quantifiers, i.e.\ one derives second-order arithmetic using the former.  
However, straightforward modifications to $(\exists^3)$ can bring down the strength considerably:  Consider $(\exists f\leq_{1}1)(\varphi(f)=0)$ and note that if $N^{0}$ is a modulus of uniform continuity on Cantor space for $\varphi$, we only need to test $2^{N}$ many\footnote{In particular, we only need to test if $\varphi(\sigma*00\dots)=0$ for all binary $\sigma^{0^{*}}$ such that $|\sigma|=N$.} sequences to verify if $(\exists f\leq_{1}1)(\varphi(f)=0)$ or not.  Now, $\MUC(\Omega)$ from Section \ref{indie} provides such a modulus, and it is thus obvious to compute (via a term of G\"odel's $T$) $\kappa$ from $\Omega^{3}$ as in $\MUC(\Omega)$.  
By Theorem \ref{proto}, $\RCAo+\WKL+(\kappa^{3})$ is conservative over $\RCA_{0}^{2}+\WKL$, which is much weaker than $(\exists^3)$.  
However, the combination of $\exists^{2}$ and $\kappa^{3}$ computes $\exists^{3}$, as shown\footnote{The proof amounts to the observation that $\N^\N$ is recursively homeomorphic to a $\Pi^0_2$-subset of Cantor space. Since this set is computable in $\exists^{2}$, any oracle call to $\exists^{3}$ can be rewritten to an equivalent oracle call to $\kappa^{3}$, in a uniform way.} by Kohlenbach in a private communication.  
\end{rem}
Secondly, we prove the following theorem to establish the claims from Remark~\ref{misfit}.    
We first show that $\NUC$ implies $\WT$, a weak fragment of \emph{Transfer}.  
\begin{thm}\label{hirko}
The system $\P+\NUC$ proves $\STP$ and also the following:
\be\tag{\textsf{\textup{WT}}}
(\forall^{\st} Y^{2})\big[ (\exists f^{1}\leq_{1}1)(Y(f)=0)\di(\exists^{\st} f^{1}\leq_{1}1)(Y(f)=0)  \big]
\ee
\end{thm}
\begin{proof}
The implication $\NUC\di \STP$ easily follows from the equivalence between $\STP$ and the normal form \eqref{coredesign} as follows:  After resolving `$\approx_{1}$', $\NUC$ implies that
\be\label{teen}
(\forall^{\st}g^{2})(\forall f^{1}, h^{1}\leq_{1}1)\big[(\forall^{\st}k)(\overline{f}k=_{0}\overline{h}k)\di g(f)=_{0}g(h)\big], 
\ee        
and bringing outside the standard universal quantifier in \eqref{teen}, we obtain 
\be\label{teenyj}
(\forall^{\st}g^{2})(\forall f^{1}, h^{1}\leq_{1}1)(\exists^{\st}k)\big[\overline{f}k=_{0}\overline{h}k\di g(f)=_{0}g(h)\big], 
\ee    
Applying idealisation as in Remark \ref{simply}, we obtain:
\be\label{froggi}
(\forall^{\st}g^{2})(\exists^{\st}k)(\forall f^{1}, h^{1}\leq_{1}1)\big[\overline{f}k=_{0}\overline{h}k\di g(f)=_{0}g(h)\big], 
\ee
and $\HAC_{\INT}$ yields (recalling again Remark \ref{simply}) standard $\Omega^{3}$ such that 
\[
(\forall^{\st}g^{2})(\forall f^{1}, h^{1}\leq_{1}1)\big[\overline{f}\Omega(g)=_{0}\overline{h}\Omega(g)\di g(f)=_{0}g(h)\big].  
\]  
Then any standard $g^{2}$ has an upper bound $\max_{|\sigma|=\Omega(g)\wedge (\forall i<|\sigma|)(\sigma(i)\leq 1)} g(\sigma*00\dots)$, and $w^{1^{*}}$ from \eqref{coredesign} is easily defined in terms of this upper bound in exactly the same way as for $\Theta(g)$ in the proof of Theorem \ref{kinkel}.  

\smallskip

For $\NUC\di \WT$, fix standard $Y^{2}$ and let $f_{0}\leq_{1}1$ be such that $Y(f_{0})=0$.  Applying $\STP$ yields standard $g_{0}\leq_{1}1$ such that $g_{0}\approx_{1}f_{0}$.      
By the uniform nonstandard continuity of $Y^{2}$, we have $0=Y(f_{0})=Y(g_{0})$, and $\WT$ follows.  
\end{proof}
Let $\MU_{2}(\kappa)$ be $(\kappa^{3})$ with the leading existential quantifier dropped.  
\begin{cor}
From `$\P\vdash \NUC\di \WT$' a term $t$ can be extracted such that 
\be\label{kloooke}
\EPA^{\omega*}\vdash(\forall \Omega^{3})\big[\MUC(\Omega)\di \MU_{2}(t(\Omega))  \big].  
\ee
\end{cor}
\begin{proof}
Note that $\NUC$ is equivalent to the normal form \eqref{froggi} while $\WT$ implies
\[
(\forall^{\st} Y^{2})(\exists^{\st} g^{1}\leq_{1}1)\big[ (\exists f^{1}\leq_{1}1)(Y(f)=0)\di(Y(g)=0)  \big].
\]
In the same way as in e.g.\ the proof of Theorem \ref{nogwel} we obtain \eqref{kloooke}.
\end{proof}
\begin{rem}\label{kito}\rm
Following the proof of the theorem, it is straightforward to define a term of G\"odel's $T$ computing the restriction of $\kappa^{3}$ to continuous functionals in terms of the classical fan functional $\Phi$ as in $\FF(\Phi)$ (and vice versa).  
\end{rem}
Thirdly, we show that the functional $\kappa$ computes a special fan functional, but not vice versa.  The former result is not such a surprise since $\kappa^{3}$ and $\exists^{2}$ together compute $\exists^{3}$ (see Remark \ref{misfit}), which in turn computes $\Theta$-functionals by Theorem~\ref{import3}.  
\begin{thm}
Any functional $\kappa$ such that $\MU_{2}(\kappa)$ computes $\Theta$ such that $\SCF(\Theta)$. 
There is no $\kappa$ as in $\MU_{2}(\kappa)$ computable in $\Theta$ such that $\SCF(\Theta)$. 
\end{thm}
\begin{proof}
For the first part, 
if $F(\kappa(F)) = 0$, we put $\Theta(F): =\{\kappa(F)\}$.
Otherwise, define $F_0(f) = F(0*f) - 1$ and $ F_1(f) = F(1*f) - 1$ and put $\Theta(F) = \Theta(F_0) \cup \Theta(F_1)$.
By the recursion theorem for Kleene (S1-S9) computability, this definition makes sense. In order to prove that it defines a total function, we need to know that for every $F$ there is an $n$ such that each binary sequence $s$ of length $n$ has at least  one extension $ f_s$ such that $F(f_s) < n$. This is a consequence of the compactness of Cantor space, and follows from $\WKL$.

\smallskip

For the second part, we note that the combination of a $\Theta$-functional with $\exists^{2}$ does not compute $\exists^{3}$, as the former are countably based, and the latter is not.  
Hence, if $\Theta^{3}$ satisfying $\SFF(\Theta)$ were to compute $\kappa^{3}$, then the combination $\Theta^{3}$ plus $\exists^{2}$ would compute the combination $\kappa^{3}$ plus $\exists^{2}$, which computes $\exists^{3}$ by Remark \ref{misfit}, yielding a contradiction.    
\end{proof}
Fourth, inspired by Remark \ref{kito}, we consider $\CC_{\ns}$ which is the modification of $\WT$ to \textbf{all} \emph{nonstandard} continuous functionals.   
Indeed, let `$Y\in C_{\ns}$' be the formula in square brackets in \eqref{hurg} restricted to binary sequences, i.e.\ expressing that $Y^{2}$ is nonstandard continuous on Cantor space.
\be\tag{$\CC_{\ns}$}
(\forall Y^{2}\in C_{\ns})\big[ (\exists f\leq_{1}1)Y(f)=0\di (\exists^{\st} g\leq_{1}1)Y(g)=0]. 
\big]
\ee
As noted above, $\WT$ is an instance of \emph{Transfer} and the move from $\WT$ to $\CC_{\ns}$ may seem like a strange one:  
one of the main `beginner mistakes' in Nonstandard Analysis is the \emph{illegal Transfer rule} (\cite{wownelly}*{p.\ 1166}) which is the incorrect application of \emph{Transfer} to formulas \emph{involving nonstandard parameters}; this often leads to contradiction.  
Despite $\CC_{\ns}$ seemingly being in violation of the \emph{illegal Transfer rule}, the former does not yield contradiction, but an old friend.  Furthermore, the condition `$Y\in C_{\ns}$' turns out to be essential, and maximal in a concrete sense.  
\begin{thm}
The system $\P$ proves $\STP\asa \CC_{\ns}$.    The system $\P_{0}+\paai$ proves that $\WT$ with the leading `st' dropped is inconsistent.  
\end{thm}
\begin{proof}
The forward implication is immediate by applying $\STP$ to the antecedent of $\CC_{\ns}$ and using the nonstandard continuity of $Y$.  
For the reverse direction, assume $\CC_{\ns}$ and suppose there is $f_{0}\leq_{1}1$ such that $(\forall^{\st}g\leq_{1}1)(f_{0}\not\approx_{1} g)$.  
Now fix some $Y^{2}\in C_{\ns}$ and nonstandard $N^{0}$, and define the functional $Z^{2}$ as follows: 
$Z(f):=0 \textup{ if } \overline{f_{0}}N=_{0}\overline{f}N$ and  $Z(f):=Y(f)+1$ $ \textup{otherwise}$.
By definition, $Z\in C_{\ns}$ has (many) zeros, but no standard one.  This contradiction yields $\CC_{\ns}\di \STP$.  

\smallskip
For the final part, consider the nonstandard functional $Y_{0}^{2}$, defined as $Y_{0}(f):=0$ if $f(N)=0\wedge (\forall i<N)(f(i)\ne0)$, and $1$ otherwise, for nonstandard $N^{0}$.  Clearly, there are many $g_{0}$ such that $Y_{0}(g_{0})=0$, but if $Y_{0}(f_{0})=0$ for \emph{standard} $f_{0}$, then $\paai$ implies that the latter is $00...$, a contradiction.   
\end{proof}
The previous nonstandard proof also gives rise to a relative computability result.  
To this end, for $\Xi^{2\di (1^{*}\times 1^{*})}\leq 1$,  let $\MU_{3}(\Xi)$ be the following formula:
\[
(\forall G^{2}, Y^{2})\big[ \PC(G, Y, \Xi(G)(1)) \di [(\exists h\leq_{1}1)(Y(h)=0)  \di (\exists h\in \Xi(G)(2))Y(h)=0]  \big], 
\]
and where $\PC(G^{2}, Y^{2}, Z^{1^{*}})$ is the formula expressing `partial continuity' as follows:
\be\label{PC}
(\forall f^{1}\in Z)(\forall g^{1}\leq_{1}1)(\overline{f}G(f)=\overline{g}G(f)\di Y(f)=Y(g)).
\ee
The following corollary establishes another nice result, namely that weakening `$Y^{2}\in C$' in the definition of the classical fan functional (see Definition~\ref{homaar}) to `partial continuity' as in \eqref{PC}, leads to a special fan functional.  
\begin{cor}
From `$\P\vdash [\STP\asa \CC_{\ns}]$', terms $t, s$ can be extracted such that 
\be\label{kloooke2}
\EPA^{\omega*}\vdash(\forall \Theta^{3})\big[\SCF(\Omega)\di \MU_{3}(t(\Theta))\big] \wedge (\forall \Xi^{3})\big[ \MU_{3}(\Xi)\di \SCF(s(\Xi)) \big].  
\ee
\end{cor}
As shown in \cite{sayo}, $\WWKL_{0}$ is equivalent to the statement that every \emph{bounded} continuous functional on the unit interval is Riemann integrable.  
We suspect that adding a boundedness condition to `$Y^{2}\in C_{\ns}$' yields an equivalence to $\LMP$.

\smallskip
Finally, we discuss the differences between $\kappa$ and $\Theta$-functionals in more detail. 
\begin{enumerate}
 \renewcommand{\theenumi}{\roman{enumi}}
\item In contrast to Kohlenbach's axioms and $\kappa$, $\STP$ and $\Theta$-functionals are not obviously instances of comprehension.   In other words, the latter are (more) `mathematical' in nature, while the former are `logical' in nature, especially in light of the intuitive interpretation just below Definition \ref{dodier}.\label{kulk}
\item 
As noted above, instances of \emph{Transfer} translate to a kind of comprehension axiom (with a dash of choice).  
However, the step from $\WT$ to $\CC_{\ns}$ seems to violate Nelson's \emph{illegal Transfer rule}, i.e.\ $\CC_{\ns}$ (and hence $\STP$) seems orthogonal to \emph{Transfer}.  One thus expects that the functionals resulting from $\STP$ are similarly orthogonal to comprehension.\label{soma}  
\end{enumerate}
Mathematical naturalness as in item \eqref{kulk} is important and worth pointing out, as it is essential to e.g.\ the Big Five phenomenon of RM, 
and the latter program is after all a main topic of this paper.   
The quest for `mathematically natural' theorems not provable in major logical systems (like the Paris-Harrington theorem and Peano arithmetic as can be found in \cite{barwise}*{D8.\S1}) should also be mentioned.  

\smallskip
It is more difficult to explain item \eqref{soma}:  The latter stems from the idea that while \emph{Transfer} corresponds (gives rise to/is translated to) to comprehension axioms with a dash of the axiom of choice, $\STP$ is \emph{fundamentally different} from \emph{Transfer}, but we do not know how to make this intuition concrete.

\section{Summary and Future Research}\label{kloi}
\subsection{Future research}
We discuss some open questions and future research.
Regarding Nonstandard Analysis, we have the following questions.  
\begin{enumerate}
 \renewcommand{\theenumi}{\roman{enumi}}
\item The axiom $\STP$ is equivalent to \eqref{fanns}, which is just $\WKL^{\st}$ for \emph{all} binary trees; the same holds for $\WWKL^{\st}$ and $\LMP$.  Most theorems from the \emph{RM zoo} (\cite{damirzoo}) can be similarly modified, but which resulting theorems have a normal form and have interesting properties?  What about $\RT_{2}^{2}$, $\ADS$ or $\textsf{EM}$ from \cite{dslice}, or $\textsf{RWKL}, \textsf{RWWKL}$ from\footnote{The authors of \cite{welcome} note that $\RWKL$ is \emph{robust}, and the same seems to hold for its nonstandard counterpart.  In particular, the robustness properties of $\STP$ and $\LMP$ discussed in Section \ref{pampson} also hold for the associated `Ramsey-type' versions.} \cite{stoptheflood, welcome}?
\item Are there any interesting principles between  $\STP$ and $\LMP$? 
\item What is the role of principles `close to' $\WWKL$, including (nonstandard versions of) \textsf{POS} from \cite{kjoske} or $n$-$\WWKL$ from \cite{avi1337}?
\item Are there other `explosions' in Nonstandard Analysis?
\end{enumerate}
Topics related to the above items have been studied in \cite{samcie18, bennosam}.  

\smallskip

Regarding computability theory, the following questions were formulated in \cite{dagsam.I.p}, and later solved in \cites{dagsamII, dagsamIII, dagcie18}.  
\begin{enumerate}
 \renewcommand{\theenumi}{\roman{enumi}}
 \setcounter{enumi}{4}
\item Do the classes of instances of $\Lambda$ and $\Theta$ have minimal objects in the sense of Kleene-degrees or other kinds of degrees of complexity?
\item Is the hyper-jump computable from $\exists^2$ and (any given instance of) $\Theta$?
\item Is $\Theta$ definable from $\Lambda$ and the hyper-jump?
\item Is Gandy's Super-jump (\cite{supergandy}) computable in a natural $\Theta$-functional and $\exists^2$?
\end{enumerate}
Regarding computability theory and Nonstandard Analysis, we have the following:  
\begin{enumerate}
 \renewcommand{\theenumi}{\roman{enumi}}
 \setcounter{enumi}{8}
\item We have observed that computability via a term of G\"odel's $T$ arises from proofs in $\P$ and vice versa.  
Is there a natural formulation of S1-S9 computability in Nonstandard Analysis?\label{lesje}
\end{enumerate}
Item \eqref{lesje} should be viewed in light of Remark \ref{bestio}.  However, it stands to reason that the problems mentioned in the latter can be solved by declaring more general type constructors (than the recursor constants) standard in an extension of $\P$.

\subsection{Summary of results}
Figure \ref{clap} below summarises our results.   
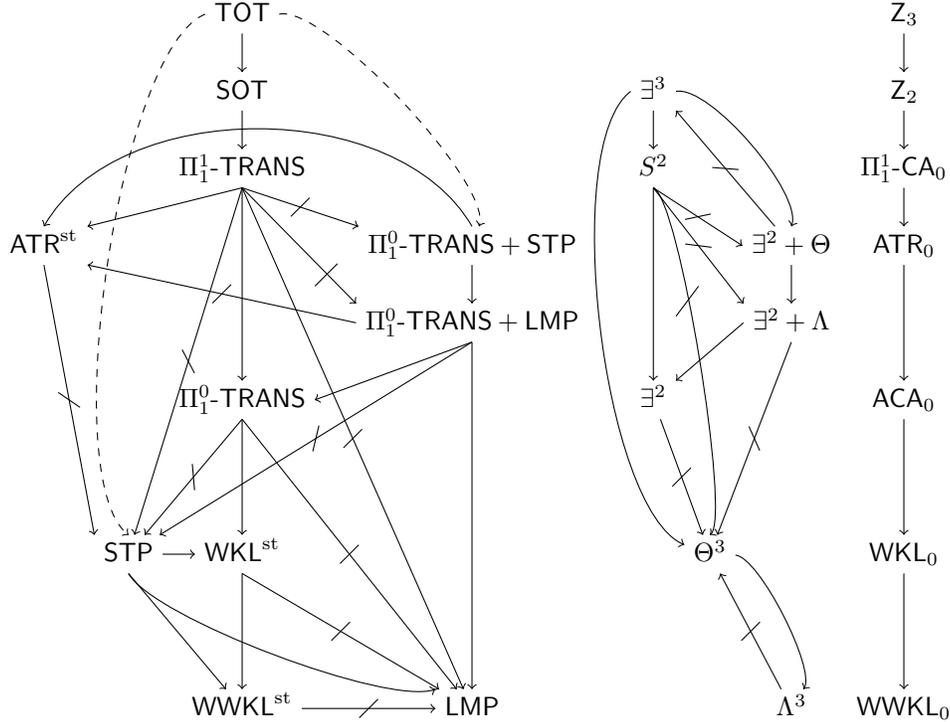
\begin{figure}[h!]
\begin{tikzpicture}[description/.style={fill=white,inner sep=2pt}]
\matrix (m) [matrix of math nodes, row sep=1.5em,
column sep=0.3em, text height=1.5ex, text depth=0.25ex]
{ ~&~&\textsf{TOT}&~&~&~&~&~&~&\textsf{Z}_{3}\\
 &~ & \textsf{SOT} &~ & ~ &~& \exists^3&~&~&\textsf{Z}_{2}\\
~& ~&\Paai &~ & ~ & ~ &S^{2}&~&~&\FIVE\\
\ATR^{\st}&~&~ &~& \paai+\STP&~& ~&&\exists^{2}+\Theta&\ATR_{0}\\ 
&&&& \paai+\LMP&&&&\exists^{2}+\Lambda&~\\
&&\paai&&&~& \exists^{2}&~& ~&\ACA_{0}\\
~ &~&~&~& ~~ &~&~&~& ~& \\
&\STP& \WKL^{\st} &~&&~&&\Theta^{3}& &\WKL_{0} \\
~ &&&& ~ \\
&& \WWKL^{\st} && \LMP&&~&&\Lambda^{3}&\WWKL_{0}\\};

\path[-]  (m-1-10) edge[->] (m-2-10);
\path[-]  (m-2-10) edge[->] (m-3-10);
\path[-]  (m-3-10) edge[->] (m-4-10);
\path[-]  (m-4-10) edge[->] (m-6-10);
\path[-]  (m-6-10) edge[->] (m-8-10);
\path[-]  (m-8-10) edge[->] (m-10-10);

  \draw[->] (m-5-9.south) -- node[ sloped, strike out,draw,-]{} (m-8-8);

\draw[->] (m-3-7.south)  -- node[sloped, strike out,draw,-]{} (m-4-9.west);
\draw[->] (m-3-7.south)  -- node[sloped, strike out,draw,-]{} (m-5-9.north west);
\path[-]  (m-4-9) edge[->] (m-5-9.north);
\path[-]  (m-6-7.north east) edge[<-] (m-5-9.west);
\path[-]  (m-2-7) edge[->] (m-3-7);

\path[-]  (m-3-7) edge[->] (m-6-7);
  \draw[->] (m-6-7) -- node[ strike out,draw,-]{} (m-8-8);
		\path[-]  (m-3-7.south) edge[->, bend right=300,looseness=0.2] (m-8-8.north);
\path[-]  (m-2-7.west) edge[->, bend left=300,looseness=0.5] (m-8-8);
\path[-]  (m-8-8) edge[->, bend right=300,looseness=0.4] (m-10-9);
  \draw[->] (m-10-9) -- node[ strike out,draw,-]{} (m-8-8);
 
  \path[-]  (m-2-7.east) edge[->, bend right=300,looseness=0.5] (m-4-9.north);
    \draw[->] (m-4-9) -- node[sloped, strike out,draw,-]{} (m-2-7.south east);


\path[-]  (m-1-3) edge[->] (m-2-3);
\path[-]  (m-2-3) edge[->] (m-3-3);
\path[-]  (m-1-3.west) 
		edge[->, dashed, bend left=300,looseness=0.45] (m-8-2.north)
		 (m-6-3)   edge[<- ](m-3-3); 
	\path[-](m-6-3.south)	 edge[->] (m-8-3);
		      
  \draw[->] (m-6-3.south) -- node[ strike out,draw,-]{} (m-10-5);
    \draw[->] (m-6-3.south) -- node[sloped, strike out,draw,-]{} (m-8-2);
  \draw[->] (m-3-3.south) -- node[ strike out,draw,-]{} (m-10-5);
    \draw[->] (m-3-3.south) -- node[]{} (m-5-5.north west);
        \draw[->] (m-3-3.south) -- node[]{} (m-4-1.north east);
           \draw[<-] (m-8-2.north west) -- node[sloped, strike out,draw,-]{}  (m-4-1.south);
           
    \draw[->] (m-3-3.south) -- node[sloped, strike out,draw,-]{} (m-8-2);
\path[-]	(m-8-2) edge[->] (m-8-3);
  \draw[->] (m-8-3.south) -- node[ strike out,draw,-]{} (m-10-5);
\path[-]	(m-8-3.south) edge[->] (m-10-3);
\path[-]	(m-8-2.south) edge[->] (m-10-3);
  \draw[->] (m-10-3) -- node[sloped, strike out,draw,-]{} (m-10-5);
\path[-]  (m-8-2.south) edge[->, bend right=400,looseness=0.4] (m-10-5.north west);
  \draw[->] (m-5-5.south) -- node[sloped, strike out,draw,-]{} (m-8-2);
    \path[-]  (m-5-5.south) edge[->] (m-6-3.east);
       \path[-] (m-5-5) edge[->] (m-10-5);
       \path[-]  (m-1-3.east) 
		edge[->, dashed, bend right=300,looseness=0.4] (m-4-5);
		
	\path[-](m-5-5.north) edge[<-] (m-4-5.south);
	       \path[-]  (m-4-5.north) edge[->, bend left=300,looseness=0.9] (m-4-1.north);
		\draw (2.6,0.6) -- (2.9,1);
			\draw (-2.2,1) -- (-1.9,1.3);
		    \draw[->] (m-3-3.south) -- node[ strike out,draw,-]{} (m-4-5.north west);
		        \draw[<-] (m-4-1.south east) -- node[ strike out,draw,-]{} (m-5-5.west);
\end{tikzpicture}
\caption{Summary of results}
\label{clap}
\end{figure}~\\
By way of a legend, in the right column are the linearly ordered `Big Five, with above them full second-order arithmetic $\textsf{Z}_{2}$ and below them the system $\WWKL_{0}\equiv \RCA_{0}+\WWKL$.  In the middle column, we classify the functionals studied in this paper as follows: $\RCAo$ plus the existence of the pictured functional is (at least or exactly) at the level of the corresponding system on the right; (struck out) arrows denote (non)computability.  
By `$\Psi$ computes $\Phi$' we mean that all instances of $\Psi$ can compute (in a uniform way) an instance of $\Phi$.
By contrast `$\Psi$ does not compute $\Phi$' means that there is an instance of $\Psi$ that cannot compute any instance of $\Phi$.

\smallskip   

In the left column, we classify the nonstandard axioms studied in this paper as follows: $\P_{0}$ plus the pictured nonstandard axioms is (at least or exactly) at the level of the corresponding system on the right; (struck out) arrows denote (non)implication over $\P_{0}$.  The dashed arrows imply implication over $\P_{0}+\X$.    

\smallskip

Our results suggest that the RM of Nonstandard Analysis is much more `wild' than the `standard' counterpart from \cite{simpson2}:  For instance, the nonstandard counterparts of the Big Five systems and $\WWKL_{0}$ are not even linearly ordered.  Similarly, the higher-order framework is much more `wild' than the second-order counterpart from \cite{simpson2}: For instance, $\Theta$ and $\Lambda$-functionals are natural variations of the usual fan functional with rather extreme computational hardness compared to their first-order strength.  The difference in strength when adding $\paai$ to $\STP$ and $\LMP$, or equivalently: when combining $\exists^{2}$ with $\Theta$-functionals and $\Lambda_{1}$, is another example of `wild' behaviour.

\smallskip
On a historical note, our results in the RM of Nonstandard Analysis should be viewed in light of the following 1966 anecdote by Friedman regarding Robinson.  
\begin{quote}
I remember sitting in Gerald Sacks' office at MIT and telling him
about this [version of Nonstandard Analysis based on PA] and the conservative extension proof. He was
interested, and spoke to A.\ Robinson about it, Sacks told me that A.\
Robinson was disappointed that it was a conservative extension. (\cite{HFFOM})
\end{quote}
In light of the previous quote, we believe Robinson would have enjoyed learning about the `new' mathematical object that is 
the special fan functional originating from Nonstandard Analysis.  As it happens, many (if not most) theorems of second-order arithmetic 
can be modified to yield similar `special' functionals with exotic computational properties.    
Thus, Figure \ref{clap} raises many questions, both in computability theory and Nonstandard Analysis, discussed in the previous section.

\begin{ack}\rm
Our research was supported by FWO Flanders, the John Templeton Foundation, the Alexander von Humboldt Foundation, LMU Munich (via their Excellence Initiative), the University of Oslo, and the Japan Society for the Promotion of Science.  
The authors express their gratitude towards these institutions. 
The authors thank Ulrich Kohlenbach, Vasco Brattka, and Anil Nerode for their valuable advice.  
Lastly, we are greatly indebted to the referees of this paper for their many helpful suggestions that have greatly improved this paper.  
\end{ack}

\end{document}